\documentclass[final,leqno]{siamltex}

\usepackage{amscd,amsmath,amssymb,verbatim}
\usepackage{graphicx}
\usepackage{cite}

\usepackage{showkeys,color}
\numberwithin{equation}{section}

\newcommand{\pd}[2]{\dfrac{\partial{#1}}{\partial{#2}}}  

\newcommand{\N}{\mathbb{N}}
\newcommand{\Q}{\mathbb{Q}}
\newcommand{\R}{\mathbb{R}}
\newcommand{\Z}{\mathbb{Z}}

\renewcommand{\a}{\alpha}
\renewcommand{\b}{\beta}
\newcommand{\g}{\gamma}
\renewcommand{\d}{\delta}
\newcommand{\e}{\epsilon}

\renewcommand{\t}{\tau}

\newcommand{\1}{\mathbf{1}}
\newcommand{\0}{\mathbf{0}}
\newcommand{\av}[1]{\left|{#1}\right|}

\newcommand{\norm}[1]{\left\|{#1}\right\|}
\renewcommand{\l}{\left}

\renewcommand{\r}{\right}

\newcommand{\be}{\begin{enumerate}}
\newcommand{\bi}{\begin{itemize}}
\newcommand{\ee}{\end{enumerate}}
\newcommand{\ei}{\end{itemize}}

\renewcommand{\P}{\mathbb{P}}

\newcommand{\K}{\mathcal{K}}
\newcommand{\ba}{\boldsymbol{\alpha}}
\newcommand{\br}{\boldsymbol{\rho}}
\newcommand{\ehat}{\widehat{e}}
\newcommand{\Spec}{\mathrm{Spec}}

\newcommand{\mcL}{\mathcal{L}}
\newcommand{\mcM}{\mathcal{M}}

\newcommand{\rmin}{\rho_{\mathsf{{min}}}}

\newcommand{\mM}{\mathbf{M}}
\newcommand{\mI}{\mathbf{I}}
\newcommand{\mD}{\mathbf{D}}
\newcommand{\vc}{\mathbf{c}}
\newcommand{\ve}{\mathbf{1}}
\newcommand{\vd}{\mathbf{d}}
\newcommand{\beps}{\boldsymbol{\epsilon}}

\newcommand{\D}[1]{D^{{#1}}}
\newcommand{\DK}[1]{D^{K,{#1}}}

\newcommand{\DL}[2]{D_\mcL^{{#1},{#2}}}
\newcommand{\DLK}[2]{D_\mcL^{K,{#1},{#2}}}

\newcommand{\DG}[2]{D_G^{{#1},{#2}}}
\newcommand{\DGK}[2]{D_G^{K,{#1},{#2}}}

\newcommand{\pDG}[2]{\partial D_G^{{#1},{#2}}}
\newcommand{\pDGK}[2]{\partial D_G^{K,{#1},{#2}}}

\newcommand{\hds}[3]{\xi^{{#1},{#2},{#3}}}
\newcommand{\hdsK}[3]{\xi^{K,{#1},{#2},{#3}}}
\newcommand{\hdscoord}[5]{\xi^{{#1},{#2},{#3}}_{{#4},{#5}}}

\newcommand{\G}[2]{G^{{#1},{#2}}}
\newcommand{\GK}[2]{G^{K,{#1},{#2}}}
\newcommand{\Gcoord}[4]{G^{{#1},{#2}}_{{#3},{#4}}}
\newcommand{\GKcoord}[4]{G^{K,{#1},{#2}}_{{#3},{#4}}}

\newcommand{\sstar}[1]{s_\star^{{#1}}}
\newcommand{\sstarK}[1]{s_\star^{K,{#1}}}

\newcommand{\psimap}[1]{\psi^{{#1}}}
\newcommand{\psimapK}[1]{\psi^{K,{#1}}}

\newcommand{\F}[1]{F^{{#1}}}
\newcommand{\HH}[3]{\mathcal{H}^{{#1},{#2},{#3}}}

\newcommand{\smin}{s_\star}

\newcommand{\xiattr}[3]{\xi_\star^{{#1},{#2},{#3}}(t)}
\newcommand{\nstar}[1]{n_\star({{#1}})}

\newcommand{\xeq}[2]{x_{\mathsf{EQ}}^{{#1},{#2}}}

\newtheorem{remark}[theorem]{Remark}

\title{Synchrony and Periodicity in an Excitable Stochastic
  Neural Network with multiple subpopulations} 

\author{Lee DeVille\thanks{Department of Mathematics, University of
    Illinois at Urbana Champaign, 1409 W. Green Street, Urbana,
    Illinois, 61801 (\tt {rdeville@illinois.edu}).}\and Yi
  Zeng\thanks{Department of Mathematics, Massachusetts Institute of
    Technology, 77 Massachusetts Avenue, Cambridge, Massachusetts
    02138 \tt(yizeng@mit.edu).}}

\begin{document}
\maketitle


\begin{abstract}
  We consider a fully stochastic excitatory neuronal network with a
  number of subpopulations with different firing rates.  We show that
  as network size goes to infinity, this limits on a deterministic
  hybrid model whose trajectories are discontinuous. The jumps in the
  limit correspond to large synchronous events that involve a large
  proportion of the network.  We also perform a rigorous analysis of
  the limiting deterministic system in certain cases, and show that it
  displays synchrony and periodicity in a large region of parameter
  space.
\end{abstract}


\section{Introduction}

The study of oscillator synchronization has made a significant
contribution to the understanding of the dynamics of real biological
systems~\cite{Buck.Buck.68, MS90, Liu.Hui.07, Peskin.75,
  Guevara.Glass.82, Glass.etal.87, Czeisler.etal.80, Tyson.etal.99,
  Ermentrout.Rinzel.96, Kapral.Showalter.book, Tyson.Keener.88,
  Cartwright.etal.99, Hansel.Sompolinsky.92, Pakdaman.Mestivier.04},
and has also inspired many ideas in modern dynamical systems theory.
See ~\cite{Sync.book,PRK.book,Winfree.book} for reviews.

The prototypical model in mathematical neuroscience is a system of
``pulse-coupled'' oscillators, that is, oscillators that couple only
when one of them ``fires''.  More concretely, each oscillator has a
prescribed region of its phase space where it is active, and only then
does it interact with its neighbors.  There has been a large body of
work on deterministic pulse-coupled networks~\cite{Kuramoto.91, GH93,
  TMS93, VAE94, BC98, TKB98, VS98, CWJ99, Knight.72, Peskin.75, MS90,
  SU00}, mostly studying the phenomenon of synchronization on such
networks.

In~\cite{BMB-DP, MMNP-DPS}, the first author and collaborators
considered a specific example of a network containing both
refractoriness and noise; the particular model was chosen to study the
effect of {\em synaptic failure} on the dynamics of a neuronal
network.  What was observed in this class of models is that when the
probability of synaptic success was taken small, the network looked,
more or less, like a stationary process, with a low degree of
correlation in time; when the probability of synaptic success was
taken large, the system exhibited synchronous behavior that was close
to periodic.  Both of these behaviors are, of course, expected: strong
coupling tends to lead to synchrony, and weak coupling tends not
to. The most interesting observation was that for intermediate values
of the coupling, the network could support both synchronized and
desynchronized behaviors, and would dynamically switch between the
two.

The main mathematical results of~\cite{MMNP-DPS} explained this
phenomenon in two ways: it first showed that in the limit
$N\to\infty$, the dynamics of the neuronal network limited onto a
deterministic dynamical system.  (What was unusual for this model was
that the deterministic system was a {\em hybrid} system: a system of
an continuous flow coupled to a map of the phase space.  The effect of
this system is to have continuous trajectories which jump at
prescribed times.)  The second part of the result was to study the
dynamics of this hybrid system, and show that in certain parameter
regimes the deterministic system was {\em multistable}, i.e. had
multiple attractors for the dynamics.  Putting these two together
explained the switching behavior observed in the finite $N$ model,
since the stochastic system would switch, on long times scales,
between the attractors that exist in the $N\to\infty$ limit.

In this paper, we consider an extension of the model where we allow
for several independent subpopulations with different intrinsic firing
rates.

\begin{figure}[ht]
\begin{centering}
  \includegraphics[width=1.0\columnwidth]{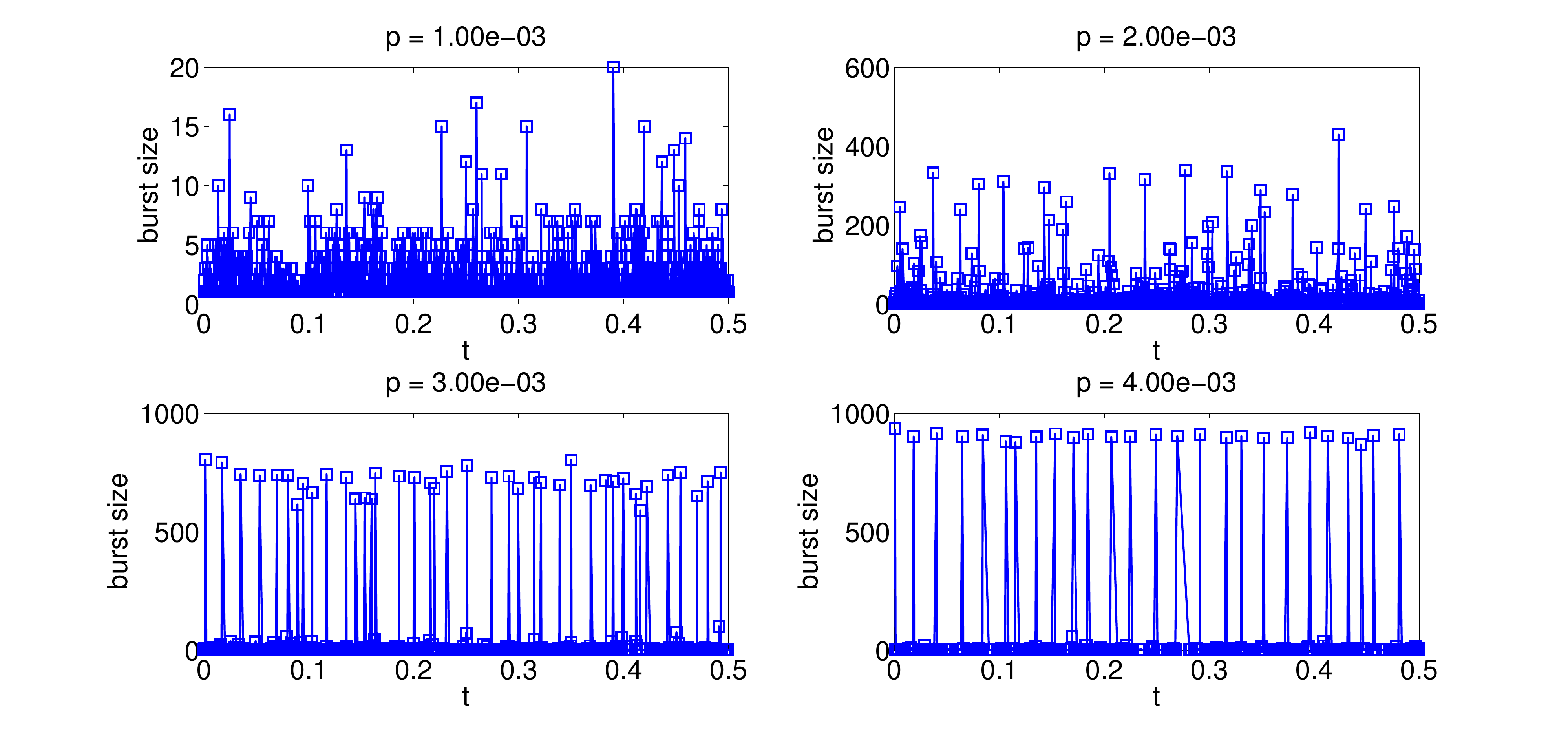}
  \caption{Different behaviors of the model.  We fix $M=10$, $N=1000$,
    and plot different dynamics of the model that correspond to
    different $p$.  As we increase $p$, we see the change from
    asynchronous and irregular behavior to synchronous and periodic
    behavior.}
  \label{fig:trace}
\end{centering}
\end{figure}

\section{Stochastic Model}

Our model is a stochastic neuronal network model with all-to-all
excitatory coupling whose details we elucidate in this section.  We
first describe the dynamics of a finite network, then describe the
``mean field'' limit of the network as we take the network size to
infinity.

\begin{figure}[ht]
\begin{centering}
  \includegraphics[width=0.6\columnwidth]{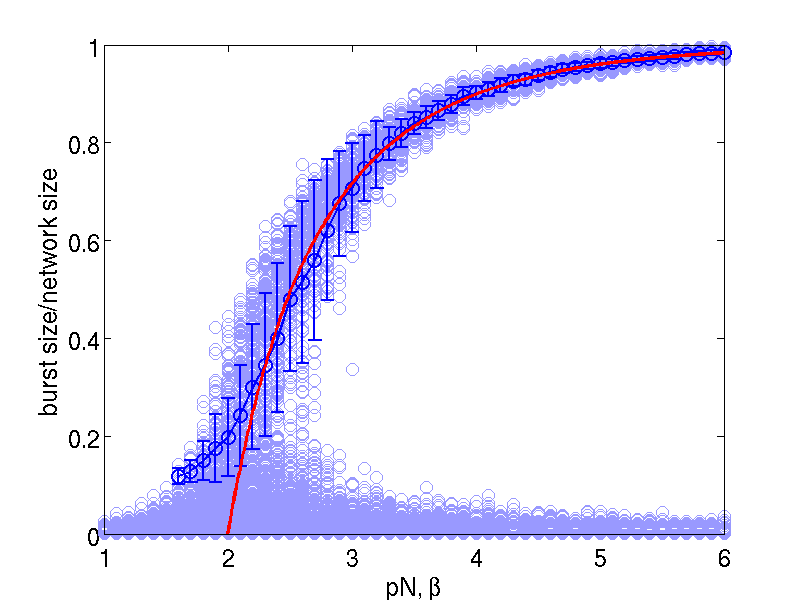}
  \caption{The meaning of the blue data: we fix a choice of $\ba$,
    $K=2$, and $N=1000$, run the stochastic neuronal network studied
    in this paper, and plot the burst sizes in light blue.  For $p$
    large enough, we also plot the mean and standard deviations of the
    burst sizes for all of the bursts larger than one-tenth the size
    of the network. In red, we plot the deterministic burst size (as a
    proportion of network size) in the deterministic limit defined in
    Section~\ref{sec:K2} below --- the main result of
    Section~\ref{sec:K2} is that the stochastic (blue) system limits
    to the deterministic (red) system as $N\to\infty$.}
  \label{fig:bvp}
\end{centering}
\end{figure}

\subsection{Fixed network size $N$} 

The state of a single neuron is given by its voltage.  We assume here
and throughout that the voltage can only take one of a finite number
of levels (see~\cite{BMB-DP} for a justification of this, but, in
short, it is a reasonable modeling assumption if we assume that the
neurons' leak is sufficiently small).  We denote these levels by $\K =
\{0,1,\dots,K-1\}$, and throughout $K$ will be the number of voltage
levels possible for each neuron.  We will assume that the network has
$N$ neurons, and thus the entire state space of the network can be
represented by $V_t = \{V_{n,t}\}_{n=1}^N \in \K^N$.  

If a neuron is ever promoted to level $K$ it is said to ``fire''.  The
effect of neuron $i$ firing is that it potentially raises the level of
every other neuron in the network; more precisely, when neuron $i$
fires it promotes neuron $j$ one level with probability $p_{ij}$, if
$j$ has not already fired.  It is clear that any one neuron firing can
lead to multiple other neurons firing, so that there can be a
``cascade'' of neuronal firings all initiated by the firing of a
single neuron.  These bursting dynamics are exactly the same as
in~\cite{MMNP-DPS}, and we describe these now.

\subsubsection{Bursting dynamics}

We now give a precise description of a burst, which we will denote by
the random partial function $B\colon\K^n\to\K^n$.  

We add two ``virtual states'' to $\K$, denoted by $Q,P$, where $Q,P$
can be thought of as a ``queue'' and as ``processed'' neurons,
respectively.  Then, let us assume that there exists a $t>0$ such that
$V_{i,t} = K$ for a single $i$.  Define the initial sets as:
\begin{equation*}
  S_{k,0} = \{j: V_{j,t} = k\},\quad Q_0 = \{i\}, \quad P_0 = \{\}.
\end{equation*}
For $i,j\in[N]$ and $u\in\N$, define the Bernoulli random variables
$\eta_{i,j}^{(u)}$ where $\P(\eta_{i,j}^{(u)} = 1) = p_{ij}$,
independently.  For each $n$, define 
\begin{equation*}
  \zeta_j^{(u)} = \sum_{i\in Q_n} \eta_{ij}^{(u)},
\end{equation*}
which represents the number of kicks that neuron $j$ receives from
neurons currently in the queue.  We then define
\begin{align*}
  S_{k,u+1} &= \bigcup_{\ell = 0}^k \{j: j\in S_{k-\ell,u} \wedge \zeta_j^{(u)}=\ell\},\quad
  Q_{u+1} = \bigcup_{\ell=0}^N \{j: j\in S_{K-\ell,u}\wedge \zeta_j^{(u)}\ge \ell\},\\
  P_{u+1} &= Q_n.
\end{align*}
In words, we promote every neuron up $\zeta_j^{(u)}$ steps, unless
this number is large enough to bring the neuron to level $K$ or above.
If this happens, we put the neuron in the queue, and we move every
neuron currently in $Q$ to $P$.

It is clear from this definition that if $Q_u = \emptyset$, then the
process stops evolving, since $\zeta_j^{(u)} = 0$ for all $j$.  Define
$u^* = \inf_{u>0} Q_u = \emptyset$ and only evolve the process for
$u=1,\dots, u^*$.  We then define
\begin{equation*}
  S_{0,u^*+1} = S_{0,u^*} + P_{u^*},\quad S_{k,u^*+1} = S_{k,u^*}, k>0,
\end{equation*}
i.e. at the end of the burst, we place all of the processed neurons
back at level 0.  We then define the map $B(V_t) =
\{B_i(V_t)\}_{i=1}^N$ by
\begin{equation*}
  B_i(V_t) = k\Leftrightarrow i\in S_{k,u^*+1}. 
\end{equation*}
This is the definition if a single component of $V_t$ is at level $K$.
If all of the components of $V_t$ are $\le k$, then we define
$B(\cdot)$ to be the identity.  If more than one component of $V_t$ is
at level $K$, we say that $B$ is undefined.  (Of course, it is clear
how we should define $B(\cdot)$ on this set, but we will see below
that this state can never occur in our dynamics, so we say that $B$ is
undefined on this state to stress this.)


\subsubsection{Non-bursting dynamics}

Here we specify what happens to the network between the bursts.  This
is where we differ from the model of~\cite{BMB-DP,MMNP-DPS}.

Choose $\br = \{\rho_n\}_{n=1}^N$, with $\rho_n>0$, and we assume that
neuron $n$ is stimulated by a exogeneous forcing of rate $\rho_n$.
More precisely, given the state $V_t$, choose $n$ independent
exponential random variables $T_n$, define $U_n = T_n/\rho_n$, and let
$n^* = \inf_{n\in[N]} U_n$.  We then say that $V_t$ is defined to be
constant on $[t,t+U_n)$, and
\begin{equation*}
  V_{t+U_{n^*}} = B(V_t + \ehat_{n^*}),
\end{equation*}
where $\ehat_k$ is the vector with a one in the $k$th slot and zero
elsewhere.  

In words, what we do is promote the level of $V_{n^*,t}$ by one level
and leave all the other neurons alone.  Recall that by the definition
of $B(\cdot)$ above, if the neuron that we have just promoted did not
reach level $K$, then we do nothing else.  If it did, then we compute
the random map $B(\cdot)$ as described above.

From this and some elementary results for Markov
chains~\cite{Norris.book}, we see that this defines a continuous-time
Markov chain for all $t>0$, once we have specified $V_0$.  It follows
from the various theorems of~\cite{Norris.book} that $n^*$ is unique,
and $U_{n^*}<\infty$ with probability one.  Moreover, by construction,
$V_t$ is a c\'{a}dl\'{a}g process, i.e. a stochastic process that is
right-continuous with countably many discontinuities.

Many readers might be familiar more with an alternative description of
the process above, where we could have said that, given the state
$V_t$ and $\Delta t \ll 1$, the probability of any neuron being
promoted is $\Delta t\cdot \sum_{n=1}^N\rho_n$, and, given that a
neuron is promoted, the probability that it is neuron $i$ is given by
$\rho_i / \sum_{n=1}^N\rho_n$, and these two are independent.  (This
is the formulation of the process sometimes called the ``Gillespie
method'' or the ``SSA method''~\cite{Gillespie.76, Gillespie.77,
  Gillespie.77.2, Gillespie.78} and it is well-known that this
definition gives rise to the same stochastic process once it is made
sufficiently precise.)

\subsection{Mean-field ($N\to\infty$) limit}

The stochastic process defined above is perfectly well-defined, but of
course for finite $N$, and for various heterogeneous choices of
$p_{ij},\rho_j$, the dynamics of this system can be quite complicated.
As is common in general in the theory of dynamical systems on
networks~\cite{HNT01,ALT06,Barrat.Berthelemy.Vespignani.book,
  Draief.Massoulie.book}, we seek to consider the limit as
$N\to\infty$ and hope that a simpler ``network level'' description can
be made.

To make the limit well-defined, we need some sort of assumptions
regarding the sets $\{p_{ij}\}_{i,j=1}^N$ and $\{\rho_n\}_{n=1}^N$.
In~\cite{BMB-DP, MMNP-DPS}, this model was considered under the
assumption that $p_{ij} = p$ and $\rho_n = \rho$, i.e. the most
homogeneous possible example was considered. It was further assumed
that $p$ satisfied the scaling law $pN \to\beta$, where $\beta>0$ is
fixed.

In this paper, we will consider the following generalization: For each
$N$, we define a partition of $N$ into $M$ disjoint sets, denoted by
$\{A_m^{(N)}\}_{m=1}^M$, and we assume that $\rho$ is constant on each
of these subsets.  We will abuse notation slightly and denote the rate
on subset $A_m^{(n)}$ as $\rho_m$.  We will then consider the limit
where $N\to\infty$ where we assume that each subpopulation scales
proportionally in the limit, i.e. we choose $\ba =
\{\a_{m}\}_{m=1}^M$, and $\br = \{\rho_m\}_{m=1}^M$, with
\begin{equation}\label{eq:defofba}
  0 < \alpha_m < 1,\quad \sum_{m=1}^M \alpha_m = 1, \quad \rho_m>0,
\end{equation}
and we assume that $||A_m^{(N)}| - \alpha_m N| < 1$ for all $m$.  (Note that
$\alpha_m N$ is not in general an integer, but we assume that $|A_m^{(N)}|$
is as close to this number as possible.)

We will then take the limit as $N\to\infty$, $pN\to\beta>0$, and
$A_m^{(N)}$ scaling as described above.  The case where $M=1$ is when
all neurons have the same rate, and is thus equivalent to the model
studied in~\cite{MMNP-DPS}.

\subsubsection{Definition of mean-field limit}

The convention that we will use below is as follows: we will always
use capital roman letters for stochastic processes, and small Greek
letters for deterministic processes.  Moreover, superscripts will
always correspond to parameters and subscripts with always correspond
to coordinates (or time in the case of as stochastic process).

To define the mean-field limit, let us consider the stochastic process
defined above with fixed $N,K,p,\ba,\br$, and let us define the
auxiliary process $X$ defined by
\begin{equation*}
  X^{(N,K,p,\ba,\br)}_{k,m,t} = \# \{n: V_{n,t} = k\ \&\ n\in A_m^{(N)}\},
\end{equation*}
i.e. $X_{k,m,t}$ counts the number of neurons at level $k$ in the set
$A_m^{(N)}$.  It is not difficult to see that all of the information
needed to evolve the process is contained in the $X$'s.  Notice that
$X_{k,m,t}\in\Z^{KM}$ for each $t$, and we index it by $(k,m)$.

We now define a deterministic {\em hybrid} dynamical system.  What
makes this system a hybrid system is that there will be two
evolutionary rules for the process defined on two different parts of
the phase space.

\begin{definition}
  Let $\ba$ satisfy~\ref{eq:defofba}.  We define 
  \begin{equation*}
    \DK \ba := \left\{x = \{x_{k,m}\}_{k,m}\in\R^{KM} : \sum_{k=0}^{K-1} x_{k,m} = \alpha_m\right\},
  \end{equation*}
  and  write $\DK\ba$ as the disjoint union $\DK\ba = \DLK
  \ba\b \dot\cup \DGK \ba\b$, where
\begin{equation*}
  \DGK\ba\b := \left\{x\in \DK\ba : \sum_{m=1}^M x_{K-1,m} \ge \frac 1\b\right\}, \quad \DLK\ba\b=\DK\ba\setminus\DGK\ba\b.
\end{equation*}
We will also write $\pDGK\ba\b$ for the set of $x$ with $\sum_{m=1}^M x_{K-1,m} = \b^{-1}$.
\end{definition}

\begin{definition}\label{def:mf}
  We now define a deterministic hybrid dynamical system
  $\hdsK\ba\br\b(t)$ with state space $\DK\ba$.  The system will be
  hybrid since it will have two different rules on  two different parts of the phase space.
%

  $\bullet$ if $\xi\in \DLK \ba\b$, i.e. $\sum_{m=1}^M \xi_{K-1,m}(t) <
  \beta^{-1}$, then we flow by
  \begin{equation}\label{eq:mfODE}
    \frac{d}{dt} \xi_{k,m}(t) = \rho_m \mu(\xi) (\xi_{k-1,m}(t)-\xi_{k,m}(t)),
  \end{equation}
  where $\mu(\xi)$ is a {\em scalar} function that we define below and
  we interpret the index modulo $K$; more specifically, if we define
  the matrix $\mcL$ by
  \begin{equation}\label{eq:defofL}
    \mcL_{(k,m),(k',m')} = \begin{cases} -\rho_m, & k=k',m=m',\\ \rho_m,& k' = k+1\pmod K, m'=m,\\ 0,&\mbox{otherwise.}\end{cases} 
  \end{equation}
  then on $\DLK\ba\b$ we flow according to $\dot \xi = \mu(\xi)\mcL x$.
  Notice that since $\mu(\xi)$ is scalar, the trajectories of the flow
  coincide with the trajectories of the flow $\dot \xi = \mcL \xi$ and
  differ only up to a time change.  (Ergo, if we are interested in the
  trajectories of the system we can ignore $\mu(\xi)$ altogether.)

  $\bullet$ We now define a map $\GK\ba\b$ with domain $\DGK\ba\b$.
  We first define
  \begin{equation*}
    \psimapK \b(\xi,s) := -s + \sum_{i=1}^K\sum_{m=1}^M \xi_{K-i,m} \left(1-\sum_{j=0}^{i-1} \frac{s^j\b^j}{j!}e^{-s\b}\right)
  \end{equation*}
  and
  \begin{equation*}
    \sstarK \b(\xi) = \inf_{s>0} \{s\colon\psimap K \b(\xi,s) = 0\}.
  \end{equation*}
  Let us now index $\R^{MK+1}$ by $(k,m)$ with $k\in\K$ and $m\in[M]$,
  plus the state $Q$, and we define the matrix $\mcM$ whose components
  are given by
  \begin{equation*} 
    \mcM_{z,z'} = \begin{cases} -1,& z=(k,m),z'=(k',m'),k=k',m=m',\\ 1,&z=(k,m),z'=(k',m'),k'=k+1,m'=m,\\1,&z=(K-1,m),z'=Q\\0,&\mbox{else}.\end{cases}
  \end{equation*}

We then define $\GK \ba \b$ componentwise by
\begin{equation}\label{eq:mfmap}
\begin{split}
  \GKcoord\ba\b k m(\xi) &=   (e^{\b s^*\mcM}\xi)_{k,m} ,\quad k=1,2,\dots,K-1,\\
  \GKcoord\ba\b0m(\xi) &= \alpha_m - \sum_{k=1}^K(e^{\b s^*\mcM}\xi)_{k,m}.
\end{split}
\end{equation}
(The final condition guarantees that $\GK\ba\b(\xi)\in\DK\ba$.  The
interpretation of this is that we redistribute all of the neurons that
have fired back to level 0, and we do so in such a way to conserve the
total number of neurons in each subpopulation.)

$\bullet$ Finally we combine these to define the hybrid system for all
$t>0$.  Fix $K,\ba,\br,\b$.  Assume $\xi(0)\in D_\mcL$, and define
\begin{equation}\label{eq:defoftau1}
  \t_1 = \inf_{t>0} \{e^{t\mcL}\xi(0) \in D_G\}.
\end{equation}
We then define 
\begin{equation*}
  \xi(t) = e^{t\mcL}\xi(0)\mbox{ for }t\in[0,\t_1),\quad \xi(\t_1) = G(e^{\t_1\mcL}\xi(0)).
\end{equation*}
(Of course, it is possible that $\t_1 = \infty$, in which case we have
defined the system for all positive time, otherwise we proceed
recursively.)  Now, given $\t_n<\infty$ and $\xi(\t_n)\in D_\mcL$,  
define
\begin{equation}\label{eq:defoftaun1}
  \t_{n+1} = \inf_{t>\t_n} \{e^{(t-\t_n)\mcL}\xi(\t_n) \in D_G\},
\end{equation}
and 
\begin{equation*}
  \xi(t) = e^{(t-\t_n)\mcL}\xi(\t_n)\mbox{ for }t\in[\t_n,\t_{n+1}),\quad \xi(\t_{n+1}) = G(e^{(\t_{n+1}-\t_n)\mcL}\xi(\t_n)).
\end{equation*}
If $\t_n = \infty$ then we define $\t_{n+1} = \infty$ as well.  We
call the times $\t_1,\t_2,\dots$ the {\bf big burst times}, and we call
$\sstarK\b(\xi(\t_n))$ the {\bf size of the big burst}.
\end{definition}

\begin{remark} 
  We note that the definition given above is well-defined and gives a
  unique trajectory for $t\in[0,\infty)$ if and only if we know that
  $G(\xi)\in D_\mcL$ for any $\xi\in D_G$.  We will show below that
  this is the case.  We will also see below that some trajectories
  have infinitely many big bursts, and some have finitely many---this
  depends both on parameters and initial conditions.
\end{remark}

\subsubsection{Intuition behind definition}\label{sec:intuition}

This is no doubt a complicated description, but all of the pieces of
this definition can be well-motivated.  The way to think of this model
is it is either ``subcritical'' or ``supercritical'': when the model
is subcritical, we move according to the flow $e^{t\mcL}$, and when it
is supercritical, we apply the map $G(\cdot)$.  The third part of the
definition is the way the two pieces are stitched together, and looks
a bit complex, but basically boils down to ``flow until the trajectory
hits a distinguished set; when it does, apply the map $G(\cdot)$, and
if it never does, just flow forever''.

The criticality parameter is the sum $\sum_{m=1}^M x_{K-1,m}$, which
is the number of neurons at level $K-1$ across all subpopulations.
The criticality threshold is $\b^{-1}$.  To see why this should be so,
we can think of a branching process description of the growth of the
queue.  Note that all neurons act the same (the only difference
between different subpopulations is the rate $\rho_m$, which only
affect interburst dynamics).  Also note that every time we process a
neuron in the firing queue, it will promote, on average $px_{K-1,m} N$
neurons from state $x_{K-1,m}$ up to firing.  In this scaling, this is
$\beta x_{K-1,m}$, and thus the mean number of children that each
firing event creates is $\beta \sum_{m=1}^M x_{K-1,m}$.  Thus the
critical threshold is whether this number is less than, or greater
than, unity.

When the system is subcritical, whenever a neuron fires, the size of
the burst that it generates will be $O(1)$.  In the scaling where all
events are $O(1)$ inside of a $O(N)$ network, the Markov chain will be
well-approximated by the mean-field differential
equation~\cite{Kurtz.72, SW}, and this is exactly the ODE given in the
first part above.  If we assume that $\mu(x)$ is the mean size of a
burst in this regime (again recalling that it is $O(1)$), then it is
plausible that we should imagine a ``flux'' ODE where the rate at
which neurons leave a state is proportional to the number in the
state, and the rate at which they enter a state is proportional to the
size of the bin corresponding to neurons with voltage one level down,
and this is~\eqref{eq:mfODE}.

When the system is supercritical, there is a positive probability of a
burst taking up $O(N)$ neurons, or, an $O(1)$ proportion of the entire
network.  The description above is meant to capture the size of this
burst, $\sstarK\b N$, and its effect.  To understand this description,
let us consider the case where we have processed $sN$ neurons,
i.e. each neuron in the network has been given $sN$ possibilities to
be promoted, each with probability $p$.  Thus, each neuron in the
network will have receive a number of kicks given by the binomial
random variable with $sN$ trials each of which with $\beta/N$
probability of success, and it is well known that in the limit as
$N\to\infty$, this binomial converges to $Po(s\b)$, a Poisson random
variable with mean $s\b$.  If we rewrite the definition of
$\psimapK\b$ in this light, we have
\begin{equation*}
  \psimapK\b(x,s) := -s + \sum_{i=1}^K\sum_{m=1}^M x_{K-i,m} \P(Po(s\b) \ge  i),
\end{equation*}
so we see that $\psimapK\b(t)$ is the expected proportion of neurons
in the queue at the time when $sN$ neurons have been processed, and
thus $\sstarK\b$ is the first time that this is zero.  Of course, this
is only an expectation, but in fact one can show that this (random)
$\sstarK\b$ satisfies a large deviation principle and $o(1)$ far away
from its mean in this scaling.  Moreover, the map $G$ is the expected
value of the system after a burst of size $\sstarK\b$; to see this,
notice that the components of $\b \mcM$ describes the average effect of
processing one neuron from the queue.

Said another way, consider the matrix exponential $e^{\b s \mcM}$ in
the definition.  This is, of course, the solution of the differential
equation $d\xi/ds = \b \mcM \xi$. One can compute explicitly that,
given an initial condition $\xi(0)$, the solution of this ODE
satisfies
\begin{equation*}
  \xi_{k,m}(s) = e^{-s\b} \sum_{j=0}^k \frac{s^j\b^j}{j!} \xi_{k-j,m}(0),
\end{equation*}
and similarly that 
\begin{equation*}
  \xi_Q(s) = \sum_{i=1}^k \sum_{m=1}^M \xi_{K-i,m}\left(1-\sum_{j=0}^{i-1}\frac{s^j\b^j}{j!}e^{-s\b}\right) = \psi_{\b,K}(\xi(0),s) + s.
\end{equation*}
Another interpretation of this definition is that we consider the flow
where mass is moving from $\xi_{k,m}$ to $\xi_{k+1,m}$ with rate $\b$,
for all $k=0,1,\dots,K-2$, and the mass from each $\xi_{K-1,m}$ is
moving to $Q$ with rate $\b$ as well.  If we then assume that mass is
leaving $Q$ at rate $1$, then $\psimapK\b(s)$ is exactly the size of
$Q$, and thus $\sstarK\b$ is the time $s$ at which $\xi_Q(s) = s$
under the $\mcM$ flow; said another way, if the mass is leaking out of
the queue at rate one, then this would be the time when the queue is
first zero.

From this argument, it is not hard to see that if $x\in\DG\ba\b$, then
$G(x)\in\DL\ba\b$.  To see this, note that since $\xi_Q(s)>s$ for some
$s>0$, at the first time when $\xi_Q(s) =s$, we must have
\begin{equation*}
  1 > \frac{d}{ds}\xi_Q(s) = \b\sum_{m=1}^M z_{K-1,m}.
\end{equation*}

It is worth noting that the mean-field system is discontinuous
whenever the map $G_\b$ is applied and continuous otherwise.  Since
this discontinuity refers to the instantaneous change in the network
when a very large event effects a significant proportion of the
network, we will also call these discontinuities ``big bursts''. We
will see below that there are some initial conditions that lead to
infinitely many big bursts, and others that lead to only finitely many.

This process is a generalization of, and quite similar to, the more
homogeneous process considered in~\cite{MMNP-DPS}; readers more
interested in the intuition behind this mean-field definition can read
Section 3 of that paper.

\subsubsection{Convergence Theorem for mean field limit}

In this section, we give a precise statement of the convergence
theorem of the stochastic neuronal network to the mean-field limit.
In the interests of space, we do not give a full proof of the theorem
here, but refer the reader to~\cite{MMNP-DPS}; it is not difficult to
see that the same proof as given there will follow with minimal
technical changes.

\begin{theorem}\label{thm:1}
  Consider any $x\in D_{K,\ba}\cap \Q^{KM}$.  For $N$ sufficiently
  large, $Nx$ has integral components and we can define the neuronal
  network process $X^{N,K,\ba,\br,p}_t$ as above, with initial
  condition $X_0^{N,K,\ba,\br,p} = Nx$.

  Choose and fix $\e,h,T>0$.  Let $\hdsK\ba\br\b(t)$ be the solution
  to the mean-field defined in Definition~\ref{def:mf} with initial
  condition $\hdsK\ba\br\b(0)=Nx$.  Define the times $\t_1,\t_2,\dots$
  at which the mean field is discontinuous, and define $b_{min}(T) =
  \min \{\sstarK\b(\xi(\t_k)):\t_k < T\}$, i.e. $b_{min}$ is the size
  of the smallest big burst which occurs before time $T$, and let
  $m(T) = \arg\max_k \t_k<T$, i.e. $m(T)$ is the number of big bursts
  in $[0,T]$.

  Pick any $\g<b_{min}(T)$.  For the stochastic process
  $X^{N,K,\ba,\br,p}_t$ and denote by $T^{(N)}_k$ the (random) times
  at which the $X^{N,K,\ba,\br,p}_t$ has a burst of size larger than
  $\g N$.  Then there exists $C_{0,1}(\e)\in[0,\infty)$ and
  $\omega(K,M)\ge 1/(M(K+3))$ such that for $N$ sufficiently large,
  \begin{equation}\label{eq:thm1.bt}
    \P\l(\sup_{j=1}^{m(T)} \av{T^{(N)}_j - \t_j} > \e\r) \le C_0(\e) Ne^{-C_1(\e) N^{\omega(K,M)}}.
  \end{equation}
  Moreover, if we define $\mathcal{T} := ([0,T] \setminus
  \cup_{j=1}^{m(T)} (T_j^{(N)}-\e,T_j^{(N)}+\e))$, and
  \begin{equation*}
    \varphi(t) = t-(T_j^{(N)}-\t_j)\mbox{ where }j=\max \{k\colon\tau_k<t\},
  \end{equation*}
  then
  \begin{equation}\label{eq:thm1.est}
    \P\l(\sup_{t\in\mathcal{T}} \av{N^{-1}X^{N,K,\ba,\br,p}_t - \hdsK\ba\br\b(\varphi(t))} > \e\r) \le C_0(\e) Ne^{-C_1(\e) N^{\omega(K,M)}}.
  \end{equation}
\end{theorem}

In summary, the theorem has two main conclusions about what happens if
we consider a stochastic neuronal network with $N\gg1$ and if we
consider the corresponding deterministic mean-field system: first, the
times of the ``big bursts'' do in fact line up, and, second, as long
as we are willing to excise a small amount of time around these big
bursts, the convergence of the stochastic process to the deterministic
process is uniform up to a rescaling in time.  In short, it is
sufficient to consider the deterministic system when the network has a
sufficient number of neurons.

The guaranteed rate of convergence is subexponential due to the
presence of the $\omega(K,M)$ power in the exponent, but note that the
convergence is asymptotically faster than any polynomial.  Numerical
simulations done for the case of $M=1$ were reported in~\cite{BMB-DP}
show that $\omega(K,1)$ seemed to be close to 1, and in fact did not
seem to decay as $K$ was increased, suggesting that the lower bound is
pessimistic and that the convergence may in fact be exponential.
However, the lower bound given in the theorem above seems to be the
best that can be achieved by the authors' method of proof.  For the
details comprising a complete proof of Theorem~\ref{thm:1},
see~\cite{MMNP-DPS}.

\section{Properties of Mean Field Model for $K=2$}

A portion of the results of~\cite{MMNP-DPS} were an analysis of the
deterministic hybrid system $\hdsK\ba\br\b(t)$ when $M=1$.  It was shown
there that one could compute the solution of the system analytically
for $K=2, \b>0$, but an analytic solution seemed intractable for $K\ge
3$.  We want to extend the analysis done there for the $K=2$ case, but
for $M>1$.

One of the main results of the deterministic analysis was that the
hybrid system was multistable for $K$ sufficiently large: more
specifically, it was shown that for all $\b< K$, there was an
attracting fixed point in the model, and for all $\b > \theta(K)$,
there was an attracting periodic orbit, and finally that for some
$\gamma\in(0,1)$, $\theta(K)<\g K$ for $K$ sufficiently large.  Thus
there was a parameter regime with multiple stable solutions, which
meant that the stochastic system would switch between these attractors
on long timescales.  The analysis performed in~\cite{MMNP-DPS} was not
able to produce a closed-form solution but used asymptotic matching
techniques.

It should be noted that the analysis of this hybrid system is quite
difficult. As is well known, the analysis of hybrid systems can be
exceedingly complicated~\cite{Branicky.97,Liberzon.Morse.99};
questions just about the stability of fixed points is much more
complicated than in the non-hybrid (flow or map) case, and stability
of periodic orbits are more complicated still.  As we see below, the
state of the art technique for this kind of problem is very
problem-specific.

\subsection{Main result}

Note: here and in the following, we will only be considering the case
of $K=2$, so that we will drop the $K$ from the notation.

\begin{theorem}\label{thm:main}
  Choose $\ba,\br$ as in~\eqref{eq:defofba}, and recall the definition of
  the hybrid system $\hds\ba\br\b(t)$ given in
  Definition~\ref{def:mf}.  For $K=2$, and any $M\ge 1$, there
  exists $\b_M>0$ such that, for $\b>\b_M$, the hybrid system has a
  globally attracting limit cycle $\xiattr\ba\br\b$.  This orbit
  $\xiattr\ba\br\b$ undergoes infinitely many big bursts.  Moreover,
  $\limsup_{M\to\infty}\b_M/\log(\sqrt M)\le1$.
\end{theorem}

We delay the formal proof of the main theorem until after we have
stated and proved all of the auxiliary results below, but we give a
sketch here.

The main analytic technique we use is a contraction mapping theorem,
and we prove this in two parts.  We first show that for any two
initial conditions, the flow part of the system stretches the distance
between them by no more than $1+\sqrt{M}/2$
(Theorem~\ref{thm:flow}). We then show that the map $G_\b$ is a
contraction, and, moreover, its modulus of contraction can be made as
small as desired by choosing $\b$ large enough
(Theorem~\ref{thm:map}).  The stretching modulus of one ``flow, map''
step of the hybrid system is the product of these two numbers, and as
long as this is less than one we have a contraction.  Finally, we also
show that for $\b>2$, there exists an orbit with infinitely many big
bursts (Lemma~\ref{lem:ibb})---in fact, we show the stronger
result that all initial conditions give an orbit with infinitely many
big bursts.  All of this together, plus compactness of the phase
space, implies that this orbit is globally attracting.

We would like to point out that several of the steps mentioned above
seem straightforward at first glance, but are actually nontrivial for
a few reasons.  

First, consider the task of computing the growth rate for the flow
part of the hybrid system.  Clearly $e^{t\mcL}\cdot$ is a linear
contraction, since its eigenvalues are
\begin{equation*}
  \{0^M,-2\rho_1,-2\rho_2,\dots,-2\rho_M\},
\end{equation*}
and the point in the nullspace is unique once $\ba$ is chosen.
However, even though the linear flow $e^{t\mcL}$ is contracting, and
clearly $\av{e^{t\mcL} x - e^{t\mcL} x'} < \av{x-x'}$ for any fixed
$t>0$, the difficulty is that two different initial conditions can
flow for a different interval of time until the first big burst, and
clearly we cannot guarantee that $e^{t\mcL}x$ and $e^{t'\mcL}x'$ are
close at all.  For example, consider the extreme case where the flow
$e^{t\mcL}x$ hits the set $D(G_\b)$ at some finite time, and the flow
$e^{t\mcL}x'$ never does---then these trajectories can will end up
arbitrarily far apart, regardless of the spectrum of $\mcL$!  Because
of both these reasons, we cannot simply use the spectral analysis of
$\mcL$ for anything useful and have to work harder at establishing a
uniform contraction bound.

Moreover, we point out another subtlety of hybrid systems, which is
that the composition of two stable systems is not stable in general.
In fact, establishing stability properties for hybrid systems, even
when all components are stable and linear, is generally a very
nontrivial problem (see, for example~\cite{Lagarias.Wang.95}.  We get
around this by showing the subsystems are each contractions (i.e. we
show that $\norm{\cdot}_2$ is a strict Lyapunov function for the
system) but the fact that we have to control every potential direction
of stretching adds complexity to the analysis.

\subsection{Simplified description of the model when $K=2$}\label{sec:K2}

As mentioned above, we only consider the case $K=2$ in the sequel, but
we let $M$ be arbitrary.  This simplifies the description of the
hybrid system significantly, and we find it useful to explicitly
derive the formulas for $K=2$ now.

We first derive the form of~(\ref{eq:mfODE},~\ref{eq:mfmap}) for the
case $K=2$ but $M$ arbitrary (we will drop the $\xi$ notation and use
$x$ throughout this section).  We will also use the notation
\begin{equation*}
  y_k = \sum_{m=1}^M x_{k,m}
\end{equation*}
to represent all neurons at level $k$, regardless of subpopulation.
The flow becomes
\begin{equation}\label{eq:ODE2} 
  \frac{d}{dt}\left(\begin{array}{cc} x_{0,m} \\ x_{1,m} \\\end{array}\right) = \rho_m\left(\begin{array}{cc} -1&1 \\ 1&-1 \\\end{array}\right)\left(\begin{array}{cc} x_{0,m} \\ x_{1,m} \\\end{array}\right).
\end{equation}
It is not hard to see that the solution to this flow is 
\begin{equation*}
x_{1,m}(t) = \frac{x_{0,m}(0)+x_{1,m}(0)}2 + \frac{x_{1,m}(0)-x_{0,m}(0)}2 e^{-2\rho_m t}, x_{0,m}(t) = \alpha_m - x_{1,m}(t).
\end{equation*}
Using the fact that $x_{0,m}(0)+x_{1,m}(0) = \alpha_m$, this simplifies to
\begin{equation}\label{eq:soln1m}
  x_{1,m}(t) = \frac{\alpha_m}2 + \frac{x_{1,m}(0)-x_{0,m}(0)}2 e^{-2\rho_m t} = \frac{\alpha_m}2 - \l(\frac{\alpha_m}2 - x_{1,m}(0)\r)e^{-2\rho_m t}.
\end{equation}

Similarly, the function $\psimap\b$ can be simplified as
\begin{equation}\label{eq:psisimple}
\begin{split}
  \psimap\b(x,s)
  &= -s + \sum_{m=1}^M x_{1,m}\left(1-e^{-s\b}\right) +  \sum_{m=1}^M x_{0,m}\left(1-e^{-s\b}-s\b e^{-s\b}\right)\\
  &= -s + y_1 \left(1-e^{-s\b}\right) + y_0\left(1-e^{-s\b}-s\b e^{-s\b}\right).  
\end{split}
\end{equation}
and recall that 
  \begin{equation*}
    \sstar\b(x) = \inf_{s>0}\psimap\b(x,s).
  \end{equation*}

\begin{proposition}\label{prop:sstar}
  $\sstar\b(x)$ is constant on any $\pDG\ba\b$, and its value depends
  only on $\b$.  We write $\smin(\b)$ for its value on this set.
  $\smin(\b)$ is an increasing function of $\b$, and
  \begin{equation*}
    \lim_{\b\to\infty} \smin(\b) = 1.
  \end{equation*}
\end{proposition}

\begin{proof}
  We see from~\eqref{eq:psisimple} that $\psimap\b(x,s)$, and thus
  $\sstar\b(x)$, depend on $x$ only through the sums $y_0$ and $y_1$.
  By definition $y_0$ and $y_1$ are constant on $\pDG\ba\b$, and
  therefore $\sstar\b(\cdot)$ is as well.  On $\pDG\ba\b$, $y_0 =
  (\b-1)/\b$ and $y_1 = 1/\b$, so on this set we can ignore $x$ and
  simplify $\psi$ to
  \begin{equation}\label{eq:psi2}
    \psimap\b(s) = 1 - s - e^{-s\b} - \frac{\b-1}{\b}s\b e^{-s\b} = 1-s -((\b-1)s+1) e^{-s\b}.
  \end{equation}
  It follows from this formula that
  \begin{equation*}
    \psimap\b(0) = 0,\quad \psimap\b(1) = -\b e^{-\b} < 0,\quad \frac{d\psimap\b}{ds}(0) = 0,\quad \frac{d^2\psimap\b}{ds^2}(0) = \b(\b-2).
  \end{equation*}
  If $\b<2$, then $\psimap\b(s)$ is negative for some interval of
  $s$ around zero, and thus $\smin(\b) = 0$.  If $\b>2$, then the
  graph $\psimap\b(s)$ is tangent to the $x$-axis at $(0,0)$ but is
  concave up, and thus positive for some interval of $s$ around zero,
  and therefore $\smin(\b) > 0$.  Since $\psimap\b(1) < 0$, it is
  clear that $\smin(\b)<1$.  Taking $\b$ large, we see that
  $\psimap\b(s) \approx 1-s$, so that $\smin(\b) \approx 1$ for $\b$
  large.

  Finally, thinking of $\psimap\b(s)$ as a function of both $s$ and
  $\b$, we have
  \begin{equation*}
    \pd{}s\psimap\b(s) =e^{-s\beta}\left(1-e^{s\beta}+\b(\beta-1) s \right),\quad
    \pd{}{\beta}\psimap\b(s) = e^{-\beta s}(\b-1)s^2.
  \end{equation*}

  Since the second derivative of $e^{s\b} \partial\psimap\b/\partial
  s$ is always negative, this means that $\partial\psimap\b/\partial
  s$ can have at most two roots, and one of them is at $s=0$.  From
  the fact that $\psimap\b(s)$ is concave up at zero, this means that
  the single positive root of $\partial\psimap\b/\partial s$ is
  strictly less than $\smin(\b)$.  From this it follows that
  $\partial\psimap\b/\partial s|_{s=\smin(\b)} > 0$.  It is clear from
  inspection that $\partial\psimap\b/\partial \b|_{s=\smin(\b)} <0$,
  and from this and the implicit function theorem, we have
  $\partial\smin/\partial \b >0$.
\end{proof}

By definition, a big burst occurs on the set $\DG\ba\b$, where $y_1
\ge \b^{-1}$. Since the flow has continuous trajectories, it must
enter $\DG\ba\b$ on the boundary $\pDG\ba\b$, and note that on this
set, formula~\eqref{eq:psi2} is valid.

In this case, we can simplify the formula for $\G\ba\b$ as follows:
\begin{equation}\label{eq:map2}
\begin{split}
  \Gcoord\ba\b0m(x) &= \alpha_m - e^{-\b\sstar\b(x)}(\b\sstar\b(x) x_{0,m} +  x_{1,m}),  \\
  \Gcoord\ba\b1m(x) &= e^{-\b\sstar\b(x)}(\b\sstar\b(x) x_{0,m} +  x_{1,m}).
\end{split}
\end{equation}
Note that  different subpopulations are coupled only
through $\sstar\b(x)$.


\subsection{Infinitely many big bursts}

In this section, we show that for $\b>2$, all orbits of
$\hds\ba\br\b(t)$ have infinitely many big bursts.

  Let us first recall that if $x\in \DG\ba\b$, then $\G\ba\b(x)
  \in\DL\ba\b$, as was shown in Section~\ref{sec:intuition}.  In
  words, every point in the big burst domain is mapped to outside of
  the big burst domain by $G_\b$.  In terms of the hybrid system, this
  means that we never apply the map twice in a row, but always have a
  nonzero interval of flow between two big bursts.

It is apparent that the flow~\eqref{eq:ODE2} has a family of
attracting fixed points given by $x_{0,m} = x_{1,m}$, and, moreover
that $x_{0,m} + x_{1,m}$ is a conserved quantity under this flow.
Therefore, if we assume that $x_{0,m}(t) + x_{1,m}(t) = \alpha_m$ for
some $t$, then this is true for all $t$.  Under this restriction,
there is a unique attracting fixed point $\xeq\ba\b$ given by 
\begin{equation*}
  \l(\xeq\ba\b\r)_{0,m} = \l(\xeq\ba\b\r)_{1,m} = \frac{\alpha_m}2.
\end{equation*}

\begin{lemma}\label{lem:ibb}
  If $\b>2$, then $\xeq\ba\b\in\DG\ba\b$ and every initial condition
  gives rise to a solution with infinitely many big bursts.
\end{lemma}

\begin{proof}
Notice that 
\begin{equation*}
  \sum_{m=1}^M \l(\xeq\ba\b\r)_{1,m} = \sum_{m=1}^M \frac{\alpha_m}2 = \frac12.
\end{equation*}
If $\b>2$, this is greater than $\b^{-1}$; every initial condition
will enter $\DG\ba\b$ under the flow.  We can actually show something
stronger: for any fixed $\b>2$, and any initial condition
$x\in\DL\ba\b$, there is a global upper bound on the amount of time
the system will flow until it hits $\DG\ba\b$.  Let $\rmin =
\min_{m=1}^{M} \rho_m$ and note that the initial condition $x_{0,m}(0)
\le \alpha_m$ for all $m$.  Then $x_{0,m}(t) = \alpha_m e^{-\rho_m
  t}$, and we have
\begin{equation*}
  \sum_{m=1}^M x_{0,m}(t) \le \sum_{m=1}^M \alpha_m e^{-\rho_m t} \le \sum_{m=1}^M \alpha_m e^{-\rmin t} = e^{-\rmin t},
\end{equation*}
so that at some time less than $t =\rmin^{-1} \log(\b/(\b-1)),$ we
have $y_0 = 1-\b^{-1}$ and thus $y_1=\b^{-1}$.  By
existence-uniqueness and using the fact that different $m$ modes are
decoupled in the flow, any other condition must reach this threshold
at least as quickly.

Since the only way for the hybrid system to have finitely many big
bursts is that it stay in the flow mode for an infinite time, we are
done.
\end{proof}

\subsection{Growth properties of stopped flow}

The main result of this subsection is Theorem~\ref{thm:flow}, which
gives an upper bound on how much the stopped flow can stretch vectors.
First we define a certain subset on which our estimates will be nice,
and which absorbs all trajectories of the flow.

\begin{definition}
\begin{equation*}
  \F\ba := \l\{x\in\D\ba\colon x_{1,m} < \frac{\alpha_m}2\mbox{ for all }m\r\}.
\end{equation*}
\end{definition}

\begin{lemma}\label{lem:F}
  For any $\b>2$, there exists $\nstar\b$ such that for any $\br>0$,
  and any solution of the hybrid system $\hds\ba\br\b(t)$ with
  initial condition $\hds\ba\br\b(0)\in\DL\ba\b$, we have
  $\hds\ba\br\b(t)\in\F\ba$ for all $t>\t_{\nstar\b}$.
\end{lemma}

\begin{remark}
  In short, this lemma says that any initial condition will enter
  $\F\ba$ after a finite number of big bursts, and this number depends
  only on $\b$.
\end{remark}

\begin{proof}
  We will break this proof up into two steps: first, we will show that
  $\F\ba$ is absorbing; second, we will show that every initial
  condition will enter it after $\nstar\b$ big bursts.  Together,
  this will prove the lemma.

  First assume that $\hds\ba\br\b(t)\in\F\ba$, and let $\t_n$ be the
  time of the next big burst after $t$.  From~\eqref{eq:soln1m}, the
  $(1,m)$ coordinate cannot cross $\alpha_m/2$ under the flow, so
  $\hds\ba\br\b(\t_n-)\in\F\ba$.  Let us denote $x =
  \hds\ba\br\b(\t_n-)$, and, recalling~\eqref{eq:map2}, we have
  \begin{equation}\label{eq:G1m}
    \Gcoord\ba\b1m(x)= e^{-\b\sstar\b(x)}(\b\sstar\b(x) x_{0,m} +  x_{1,m}).
  \end{equation}
  This is a linear combination of $x_{0,m}\in[\alpha_m/2,\alpha_m]$
  and $x_{1,m}\in[0,\alpha_m/2]$, so we need only check the extremes.
  If we take $x_{0,m} = \alpha_m$ and $x_{1,m} = 0$, then we have
  $\Gcoord\ba\b1m(x) = ze^{-z}\alpha_m$ for some $z>0$, and
  $\sup_{z>0} ze^{-z} = 1/e$.  Considering the other extreme gives
  $\Gcoord\ba\b1m(x) = (z+1)e^{-z}\alpha_m/2$, and $\sup_{z>0}
  (z+1)e^{-z} = 1$.  In either case, we have
  $\Gcoord\ba\b1m(x)<\alpha_m/2$ and we see that $\F\ba$ is absorbing.

  Now assume that $\hds\ba\br\b(0)\not\in\F\ba$.  Since $\b>2$, it
  follows from Lemma~\ref{lem:ibb} that $\hds\ba\br\b(t)$ has
  infinitely many big bursts.  Let $x =\hds\ba\br\b(\t_1-)$, noting by
  definition that $x\in\pDG\ba\b$.  Using~\eqref{eq:G1m} and $x_{1,m}
  > \alpha_m/2$, $x_{0,m} < x_{1,m}$,
  \begin{equation*}
    \Gcoord\ba\b1m(x) < e^{-\b\sstar\b(x)}(\b\sstar\b(x)+1)x_{1,m}.
  \end{equation*}
  By Proposition~\ref{prop:sstar} and again recalling that
  $(z+1)e^{-z}<1$ for all $z>0$, this means that there is an
  $h(\b)\in(0,1)$ with
  \begin{equation*}
    \hdscoord\ba\br\b1m(\t_1) < h(\b)\cdot x_{1,m}.
  \end{equation*}

  If $h(\b)x_{1,m}<\alpha_m/2$, then we are done.  If not, notice that
  the flow generated by $\mcL$ will make the $(1,m)$ coordinate
  decrease, so it is clear that if
  $\hdscoord\ba\br\b1m(t)\not\in\F\ba$ for all $t\in[0,\t_n)$, then by
  induction $\hdscoord\ba\br\b1m(\t_n) < (h(\b))^n\alpha_m$.  Choose
  $\nstar\b$ so that $(h(\b))^{\nstar\b}<1/2$, and we have that
  $\hdscoord\ba\br\b1m(\t_{\nstar\b}) < \alpha_m/2$ and thus
  $\hds\ba\br\b(\t_{\nstar\b})\in\F\ba$.
\end{proof}

\begin{theorem}\label{thm:flow}
  Choose any two initial conditions $x(0), \widetilde{x}(0)\in\F\ba
  \cap \DL\ba\b$, and define $\t,\widetilde{\t}$ as
  in~\eqref{eq:defoftau1}.  Then
  \begin{equation*}
    \norm{e^{\t\mcL}x(0) - e^{\widetilde{\t}\mcL}\widetilde{x}(0)} \le \left(1+\frac{\sqrt{M}}2\right)\norm{x(0) -\widetilde{x}(0)},
  \end{equation*}
  i.e. for any two initial conditions, the distance at the time of the
  first big burst has grown by no more than a factor of
  $1+\sqrt{M}/2$.
\end{theorem}

\begin{proof}
  Before we start, recall that the map $\e^{\t\mcL}x$ is nonlinear in
  $x$, because $\t$ itself depends nonlinearly on $x$.
  Let $\1_M$ be the all-ones column vector in $\R^M$.  Let
  $x(0)\in\DL\ba\b$ and consider a perturbation $\boldsymbol\e =
  \{\e_m\}$ with $\sum_m \e_m = 0$, i.e. $\boldsymbol\e \in
  \1_M^\perp$, and define $\widetilde{x}(0)$ by
\begin{equation*}
  \widetilde{x}_{m,1}(0) = x_{m,1}(0) + \e_m,\quad   \widetilde{x}_{m,0}(0) = x_{m,0}(0) - \e_m.
\end{equation*}
Define $\t, \widetilde{\t}$ as the burst times associated to these
initial conditions as in~\eqref{eq:defoftau1}, and by definition, we
have
\begin{equation*}
  \sum_{m=1}^M x(\t-)_{1,m} = \sum_{m=1}^M \widetilde{x}(\widetilde{\t}-)_{1,m} = \frac1\b.
\end{equation*}
Writing $\widetilde\t = \t + \d$ and using~\eqref{eq:soln1m}, we have
\begin{equation*}
\sum_{m=1}^M \l(\frac{\alpha_m}2 - \l(\frac{\alpha_m}{2} - x_{1,m}(0)\r) e^{-2\rho_m \t}\r) = \sum_{m=1}^M \l(\frac{\alpha_m}2 - \l(\frac{\alpha_m}{2} - \widetilde{x}_{1,m}(0)\r) e^{-2\rho_m (\t+\d)}\r)
\end{equation*}
Since $\widetilde x -x = O(\e)$ and $e^{-2\rho_m\delta} = (1+O(\d))$,
we can see from this expression that the leading order terms in both
$\e$ and $\delta$ are of the same order.  Thus, Taylor expanding to
first order in $\e$ and $\delta$ and canceling gives a solution for
$\delta$:
\begin{equation}
\delta = -\frac{\sum_\ell \epsilon_\ell e^{-2\rho_\ell \t}}{2 \sum_\ell \rho_\ell\l(\dfrac{\alpha_\ell}{2} - x_{1,\ell}(0)\r) e^{-2\rho_\ell \t}}+O(\e^2).
\end{equation}

We then have
\begin{align*}
\widetilde{x}_{1,m}(\t+\d) - x_{1,m}(\t) &= \e_me^{-2\rho_m\t} -2\rho_m \l(\frac{\a_m}{2}-x_{1,m}(0)\r) \d e^{-2\rho_m\t}\\
&=\e_m e^{-2\rho_m \t} - c_m\sum_{\ell} \e_{\ell} e^{-2\rho_{\ell}\t} + O(\e^2),
\end{align*}
where
\begin{equation}
\label{eqn:defineC}
c_m = \frac{\rho_m \l(\dfrac{\alpha_m}{2}  - x_{1,m}(0)\r) e^{-2\rho_m \t} }{\sum_{\ell} \rho_{\ell} \l(\dfrac{\alpha_{\ell}}{2}  - x_{1,\ell}(0)\r) e^{-2\rho_{\ell}  \t} }.  
\end{equation}
Since $x(0) \in \F\ba$, $c_m>0$. It is then clear from the
definition that $c_m<1$. 
Writing this in
matrix form in terms of $\e$ gives
\begin{equation}
\left(
\begin{array}{c}
\widetilde x_{1,m}(\t+\delta) - x_{1,m}(\t) \\
\widetilde x_{2,m}(\t+\delta) - x_{2,m}(\t) \\
\vdots \\
\widetilde x_{2,M}(\t+\delta) - x_{2,M}(\t) \\
\end{array}
\right)  = \textbf{M}_M
\left(
\begin{array}{c}
\epsilon_1 \\
\epsilon_2\\
\vdots \\
\epsilon_N\\
\end{array}
\right) + O(\e^2),
\end{equation}

where the matrix $\mM_M$ is defined as,
\begin{equation}
\label{eq:defofM}
\mM_M = \
\left(
\begin{array}{c c c c}
e^{-2\rho_1 \t} - c_1 e^{-2\rho_1 \t} & - c_1 e^{-2\rho_2 \t}   & \cdots & - c_1 e^{-2\rho_M \t} \\
- c_2 e^{-2\rho_1 \t}  & e^{-2\rho_2 \t} - c_2 e^{-2\rho_2 \t} & \cdots & - c_2 e^{-2\rho_M \t}  \\
\vdots & \vdots & \ddots & \vdots\\
-c_Me^{-2\rho_1 \t} & -c_Ne^{-2\rho_2 \t} & \cdots & e^{-2\rho_M \t}  -c_Ne^{-2\rho_M \t}\\
\end{array}
\right),
\end{equation}
or, more compactly,
\begin{equation*}
  (\mM_M)_{ij} = -c_i e^{-2\rho_j \t} + \delta_{ij}e^{-2\rho_i \t}.
\end{equation*}
Thus, the map $e^{\t \mcL}x$ has Jacobian $\mM_M$.  Since $\mM_M$ has
zero column sums, it is apparent that $\mathbf{1}^\intercal \mM_M =
\mathbf{0}$ and thus $0\in\Spec(\mM_M)$.  Since all of the nondiagonal
entries of $\mM_M$ are bounded above by one, the standard Gershgorin
estimate implies that all of the eigenvalues of
$\sqrt{\mM_M^\intercal\mM_M}$ lie in a disk of radius $O(M)$ around
the origin, but this is not good enough to establish our result.

We can work out a more delicate bound: by the definition of $\D\ba$,
we need only consider zero sum perturbations, and so in fact we are
concerned with $\mM_M$ restricted to $\1^\perp_M$.  From this and the
fundamental theorem of calculus, it follows that
\begin{equation*}
  \norm{e^{\t\mcL}x(0) - e^{\widetilde{\t}\mcL}\widetilde{x}(0)} \le \norm{\mM_M|_{\1_m^\perp}}_2\norm{x(0) - \widetilde{x}(0)},
\end{equation*}
where $\norm\cdot_2$ is the spectral norm of a matrix
(q.v.~Definition~\ref{def:specnorm} below).  Using the bound in
Lemma~\ref{thm:MatrixNormBound} proves the theorem.
\end{proof}

\begin{definition}\label{def:specnorm}
  We define the {\bf spectral norm} of a square matrix $A$ by
\begin{equation*}
  \norm{A}_2 = \sup_{x\neq \mathbf{0}} \frac{\norm{Ax}_2}{\norm{x}_2},
\end{equation*}
where $\norm\cdot_2$ is the Euclidean ($L^2$) norm of a vector.
\end{definition}

The spectral norm of a matrix is equal to its largest singular value,
and if the matrix is symmetric, this is the same as the largest
eigenvalue.  In particular, it follows from the definition that
\begin{equation*}
  \norm{Ax}_2 \le \norm{A}_2\norm{x}_2.
\end{equation*}
%


\begin{theorem}
\label{thm:MatrixNormBound}
Let $\1_M^\perp\subseteq \R^M$ denote the subspace of zero-sum
vectors.  $\mM_M\colon \1^\perp_M\to\1^\perp_M$ since it is a zero
column sum matrix, and thus the restriction is well-defined.  Then
\begin{equation}
\norm{\mM_M|_{\1^\perp_M}}_2 < 1 + \frac{\sqrt{M}}{2}.
\end{equation}
\end{theorem}

\begin{proof}
  Let us denote $\mI_M$ to be the $M$-by-$M$ identity matrix and
  $\ve_M$ the all-one column vector in $\mathbb{R}^M$.  We will also
  define the matrix $\mD_M$ and vector $\vd_M$ by
\begin{equation*}
  \vd_M = [e^{-2\rho_1 s} ,  e^{-2\rho_2 s} , \cdots, e^{-2\rho_N s} ]^\intercal,
\end{equation*}
and $\mD_M$ is the matrix with $\vd_{M}$ on the diagonal,
i.e. $(\mD_M)_{ij} = \delta_{ij} e^{-2\rho_i \t}$.

Any vector $\mathbf{v}\in\1_M^\perp$ is in the null space of the
matrix $\ve \ve^\intercal$, and thus $(\mI_M - {M}^{-1} \ve
\ve^\intercal)\mathbf{v} = \mathbf{v}$, and $\mM_M = \mM_M(\mI_M -
{M}^{-1} \ve \ve^\intercal)$ on $\1^\perp$, so it suffices for our
result to bound the norm of $\mM_M(\mI_M - {M}^{-1} \ve
\ve^\intercal)$.

We can factorize
\begin{equation}
\mM_M  = (\mI - \vc \ve ^\intercal ) \mD_M,
\end{equation}
where the components of $\vc$ are given in~\ref{eqn:defineC}.  To see
this, we compute
\begin{align*}
  ((\mI - \mathbf{c} \ve ^\intercal ) \mD_M)_{ij} &= (\mathbf{D}_M)_{ij} - (\mathbf{c} \ve ^\intercal\mathbf{D}_M)_{ij} = (\mathbf{D}_M)_{ij} -\sum_{k} c_i \cdot 1\cdot \delta_{k,j}e^{-2\rho_j\t}\\
  &= \delta_{ij}e^{-2\rho_i \t} - c_ie^{-2\rho_j\t}.
\end{align*}

Let us first write
\begin{align*}
\mM_M
&= (\mI - \mathbf{c} \ve ^\intercal ) \mD_M  = (\mD_M - \mD_M \mathbf{c} \ve ^\intercal  + \mD_M \mathbf{c} \ve ^\intercal - \mathbf{c} \ve ^\intercal \mD_M ) \\
&= \mD_M (\mI - \mathbf{c} \ve ^\intercal ) + (\mD_M \mathbf{c} \ve ^\intercal - \mathbf{c} \mathbf{d}_M ^\intercal ),
\end{align*}
where we use the relation $\ve ^\intercal \mD_M = \mathbf{d}_M^\intercal$, and then
\begin{equation}
\label{eqn:MatrixSeparation}
\mM_M ( \mI - M^{-1} \ve \ve^\intercal) = \mD_M (\mI - \mathbf{c} \ve ^\intercal )  ( \mI - M^{-1} \ve  \ve^\intercal) +  (\mD_M \mathbf{c} \ve ^\intercal - \mathbf{c} \mathbf{d}_M ^\intercal ) ( \mI - M^{-1} \ve  \ve^\intercal) .
\end{equation}
We break this into two parts.  Using the fact that $\ve^\intercal\ve = M$, we have
\begin{equation*}
  \vc \ve^\intercal(\mI_M-M^{-1}\ve \ve^\intercal) = \mI_M \vc \ve^\intercal - M^{-1}\vc \ve^\intercal\ve \ve^\intercal = \vc \ve^\intercal -\vc \ve^\intercal =0,
\end{equation*}
and thus the first term can be simplified to
\begin{equation}\label{eq:term1}
\begin{split}
  &\mD_M (\mI - \mathbf{c} \ve ^\intercal )  ( \mI - M^{-1} \ve \ve^\intercal) \\
  &=  \mD_M (\mI - M^{-1}\ve \ve^\intercal) - \mD_M(\vc \ve^\intercal)(\mI_M-M^{-1}\ve \ve^\intercal)
  =  \mD_M (\mI - M^{-1}\ve \ve^\intercal).
\end{split}
\end{equation}

Since the matrix $M^{-1}\ve \ve^\intercal$ is an orthogonal projection
matrix with norm $1$ and rank $1$, it follows that $\mI_M - M^{-1}\ve
\ve^\intercal$ is also a projection matrix with norm $1$ and rank
$M-1$. By Cauchy-Schwarz, the norm can be bounded by
\begin{equation}
\norm { \mD_M (\mI - M^{-1}\ve \ve^\intercal)}_2 \leq \norm{\mD_M}_2 \norm{\mI_M - M^{-1}\ve \ve^\intercal}_2 = \norm{\mD_M}_2  <1.  
\end{equation}
(The last inequality follows from the fact that $\mD_M$ is diagonal
and all entries are less than one in magnitude.)

For the second term in Equation~\eqref{eqn:MatrixSeparation} and
noting that $\vd^\intercal \ve = \sum_m d_m$, we obtain
\begin{equation}
(\mD_M \vc \ve ^\intercal - \vc \vd ^\intercal ) ( \mI - M^{-1} \ve  \ve^\intercal)  = \mD_M \vc \ve ^\intercal - \vc \vd ^\intercal  - \mD_M \vc \ve ^\intercal + \frac{\vd^\intercal \ve}{M} \vc \ve ^\intercal = \vc \l(\frac{\sum_m d_m}{M} \ve - \vd\r)^\intercal.
\end{equation}
This outer product is of rank $1$, and thus it has exactly one
non-zero singular value; this singular value is the product of the
$L^2$ norms of the two vectors, and therefore

\begin{equation*}
  \norm{(\mD_M \vc \ve ^\intercal - \vc \vd ^\intercal ) ( \mI_M - M^{-1} \ve  \ve^\intercal) }_2 = \norm{ \vc }_2 \norm{ \vd - \frac{\sum_m d_m}{M} \ve }_2< 1\cdot \frac{\sqrt{M}}2.
\end{equation*}
Using Equation~\eqref{eqn:MatrixSeparation}, and the triangle
inequality gives the result.
\end{proof}

\subsection{Contraction of the big burst map}

In this section, we demonstrate that $\G\ba\b$ is a contraction for
$\b$ large enough, and, moreover, that one can make the contraction
modulus as small as desired by choosing $\b$ sufficiently large.

\begin{theorem}\label{thm:map}
  For any $M\ge 1$ and $\delta>0$, there is a $\b_1(M,\d)$ such that
  for all $\b>\b_1(M,\d)$ and $x,\widetilde{x}\in\pDG\ba\b$,
  \begin{equation*}
    \norm{\G\ba\b(x) - \G\ba\b(\widetilde{x})} \le \delta\norm{x-\widetilde{x}}.
  \end{equation*}
  In particular, by choosing $\b$ sufficiently large, we can make this
  map have as small a modulus of contraction as required.
\end{theorem}

\begin{proof}
  Let us define the vector $\beps$ by
  \begin{equation*}
    \e_m = \widetilde{x}_m - x_m.
  \end{equation*}

Since $x,\widetilde{x}$ are both in $\pDG\ba\b$, we have
$\boldsymbol\e\perp\1$.  It follows from~\eqref{eq:psisimple} that
$\nabla_{\boldsymbol\e} \psimap\b(s,x)=0$.  Recall
from~\eqref{eq:map2} that
  \begin{equation*}
    \Gcoord\ba\b1m(x) = e^{-\b \sstar\b(x)}(\b \sstar\b(x) x_{0,m} - x_{1,m}),
  \end{equation*}
  and thus
  \begin{align*}
    \nabla_{\e}\Gcoord\ba\b1m(x)
    &= e^{-\b s^*(x)}\left(-\b \nabla_{\boldsymbol{\e}}\sstar\b(x)\right)(\b \sstar\b(x) x_{0,m} - x_{1,m}) + e^{-\b \sstar\b(x)}\left(\b \nabla_{\boldsymbol{\e}}x_{0,m} - \nabla_{\boldsymbol{\e}}x_{1,m}\r)\\
      &=  e^{-\b \sstar\b(x)}(\b \sstar\b(x)(-1)-1),
\end{align*}
so
\begin{equation*}
  {\nabla_{\boldsymbol{\e}}\G\ba\b(x)} = -(e^{-\b \sstar\b(x)}(\b \sstar\b(x)+1))\1.
\end{equation*}
Note that Proposition~\ref{prop:sstar} implies that
$\b\sstar\b(x)\to\infty$ as $\b\to\infty$ for any $x$.  If we define
the function $g(z) = e^{-z}(1+z)$, then it is easy to see that
\begin{equation*}
  0 < g(z) < 1\mbox{ for }z\in(0,\infty),\quad \lim_{z\to\infty}g(z).
\end{equation*}

From this and the fundamental theorem of calculus, the result follows.
\end{proof}

\subsection{Proof of Main Theorem}

Finally, to prove the theorem, we will show that the 

\begin{definition}

We define
\begin{align*}
\HH\ba\br\b\colon \DL\ba\b&\to \DL\ba\b\\
  x&\mapsto  \G\ba\b(e^{\t\mcL}x),
\end{align*}
where $\t$ is the first hitting time defined in~\eqref{eq:defoftau1}.
\end{definition}

{\bf Proof of Theorem~\ref{thm:main}.}  If we consider any solution of
the hybrid system $\hds2\ba\br\b(t)$ that has infinitely many big
bursts, then it is clear from chasing definitions  that 
\begin{equation*}
  \hds\ba\br\b(\t_n) = \left(\HH\ba\br\b\r)^n\hds\ba\br\b(0).
\end{equation*}

$\HH\ba\br\b$ is the composition of two maps, one coming from a
stopped flow and the other coming from the map $G$.  It follows from
Theorem~\ref{thm:flow} that the modulus of contraction of the stopped
flow is no more than $1+\sqrt{M}/2$ on the set $\F\ba$ whenever
$\b>2$.  It follows from Theorem~\ref{thm:map} that we can make the
modulus of the second flow less than $\d$ by choosing $\b>\b_1(M,\d)$.
Let us define
\begin{equation*}
  \b_M := \b_1\left(M,\frac{1}{1+\sqrt{M}/2}\right),
\end{equation*}
and then by composition it follows that $\HH\ba\br\b$ is a strict
contraction on $\F\ba$.  From Lemma~\ref{lem:F}, it follows that
$\DL\ba\b$ is mapped into $\F\ba$ in a finite number of iterations, so
that $\HH\ba\br\b$ is eventually strictly contracting on $\DL\ba\b$,
and therefore $\HH\ba\br\b$ has a globally attracting fixed point,
which means that the hybrid system has a globally attracting limit
cycle.

Finally, we want to understand the asymptotics as $M\to\infty$.
Choose any $0<\g_1,\g_2<1$.  By Proposition~\ref{prop:sstar},
$\b\smin\b > \g_1\b$ for $\b$ sufficiently large, and it is clear that
$e^{-z}(z+1) < e^{-\g_2 z}$ for $z$ sufficiently large.  From these it
follows that for $\b$ sufficiently large,
\begin{equation*}
  e^{-\b\smin(\b)}(\b\smin(\b)+1)<e^{-\g_1\g_2\b}.
\end{equation*}
From this we have that $\b_M < \ln(1+\sqrt{M}/2)/\g_1\g_2$ and the
result follows.

\hfill$\square$

We have show that $\b_M$ is finite and determined its asymptotic
scaling as $M\to\infty$.  It was shown in~\cite{MMNP-DPS} that $\b_1 =
2$, and we can now show that this is the case as well for $M=2$, i.e.

\begin{proposition}
$\b_2 =2$.
\end{proposition}

\begin{proof}
  Using the previous (much more general) results, we have that
  $\norm{\mM_2|_{\1^\perp}}_2 < 3/2$.  This tells us that choosing
  $\b$ large enough that $e^{-\b\smin(\b)}(\b\smin(\b)+1) < 2/3$ is
  good enough to guarantee a contraction.  Numerical approximation
  gives a value of $\b\approx 2.48$ that will guarantee this.  In
  fact, we will go further, and show that for $M=2$, we have
  $\norm{\mM_2|_{\1^\perp}}_2 < 1$ and this will be enough to
  establish that $\b_2=2$.

\newcommand{\vv}{\mathbf{v}}

  In $\R^2$, $\1^\perp$ is a one-dimensional space spanned by
  $(1,-1)^\intercal$, and thus we need only compute the eigenvalue
  associated to this vector.  If we define $\vv =
  \mM_2\cdot(1,-1)^\intercal$ and show $\av{v_1-v_2} < 2$, then we
  have established the result.  When $M=2$, we can
  write~\eqref{eq:defofM} as
  \begin{equation}
    \mM_2
    =
    \left(\begin{array}{cc} e^{-2\rho_1\t}-c_1e^{-2\rho_1\t}&-c_1e^{-2\rho_2\t} \\ -c_2e^{-2\rho_1\t}& e^{-2\rho_2\t}-c_2e^{-2\rho_2\t}\end{array}\right),
  \end{equation}
  and thus 
\begin{equation*}
  \vv = \left(\begin{array}{cc} e^{-2\rho_1\t} - c_1 e^{-2\rho_1\t} + c_1 e^{-2\rho_2\t} \\ -c_2e^{-2\rho_1\t} - e^{-2\rho_2\t}+c_2e^{-2\rho_2\t} \\\end{array}\right).
\end{equation*}
Thus 
\begin{equation*}
  v_1-v_2 = e^{-2\rho_1\t}(1-c_1+c_2)+e^{-2\rho_2\t}(1+c_1-c_2).
\end{equation*}
Using $c_1 + c_2 = 1$, this simplifies to 
\begin{equation*}
  v_1-v_2 = 2c_2e^{-2\rho_1\t}+2c_1e^{-2\rho_2\t}.
\end{equation*}
Since it is clear that $v_1-v_2>0$, we need to show that $v_1-v_2 < 2$, or
\begin{equation*}
  2c_1 e^{2\rho_1\t} + 2c_2e^{2\rho_2\t} < 2e^{(\rho_1+\rho_2)\t}.
\end{equation*}
Writing $A = \rho_1 (\a_1/2 - x_{1,1}(0))$, $B =   \rho_2 (\a_2/2 - x_{1,2}(0))$, this becomes
\begin{equation}
\frac{A+B}{Ae^{2\rho_2 \t} + Be^{2\rho_1\t}} < 1,
\end{equation}
but this is clear since $e^{2\rho_1\t},e^{2\rho_2\t}>1$.
\end{proof}


\begin{remark}
  We conjecture from numerical evidence
  (cf.~Figure~\ref{fig:PhaseDiagram}) that, in fact, $\beta_M = 2$ for
  all $M$.  The techniques used in this paper cannot prove this,
  however.
\end{remark}

\section{Numerical simulations}

In this section, we will first present a numerical simulation of the
mean field system and compare to the full stochastic system.  We
verify the existence of a unique attracting periodic orbit for $K=2$,
as proven above.  Finally, we show numerically that this unique
attractor exists, at least for some parameter values, for $K>2$.

\subsection{ Mean field}

We first numerically solve the hybrid ODE-mapping system, with $M=3$
and random $\alpha_i, \rho_i$.  The ODE portion of the hybrid system
can be solved explicitly, and we use MATLAB's {\texttt{fsolve}} to
determine the hitting times $\tau_i$.  We plot the results for $\beta
=2.1$, $\beta = 2.5$ for a single initial condition in Figure
\ref{fig:ThreePopulations}.  We observe that each neuron population is
attracted to a periodic orbit after several bursts.

\begin{figure} [ht]
\begin{center}
\begin{tabular}{ c c}
\includegraphics[width=0.5 \columnwidth]{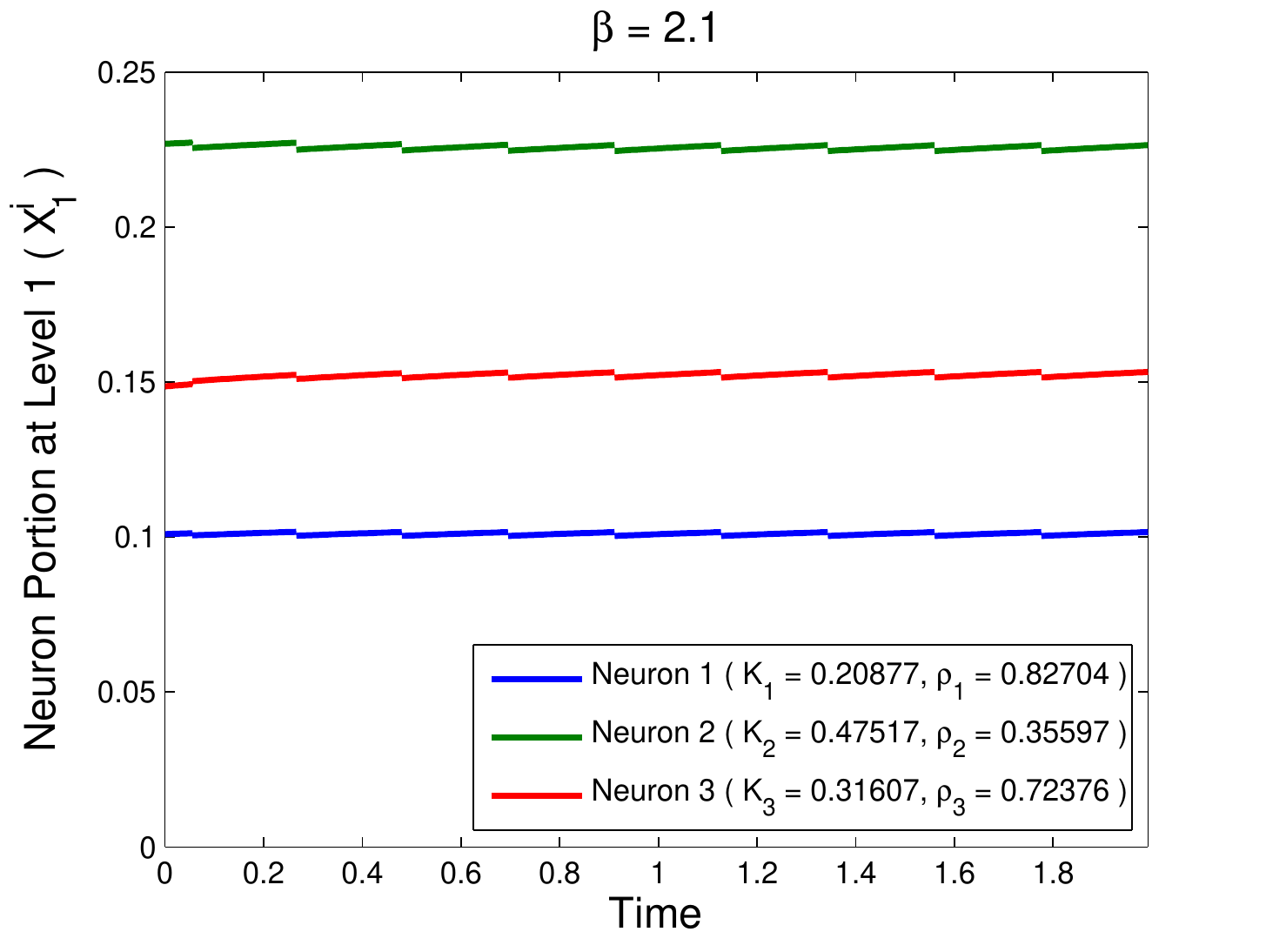} & \includegraphics[width=0.5 \columnwidth]{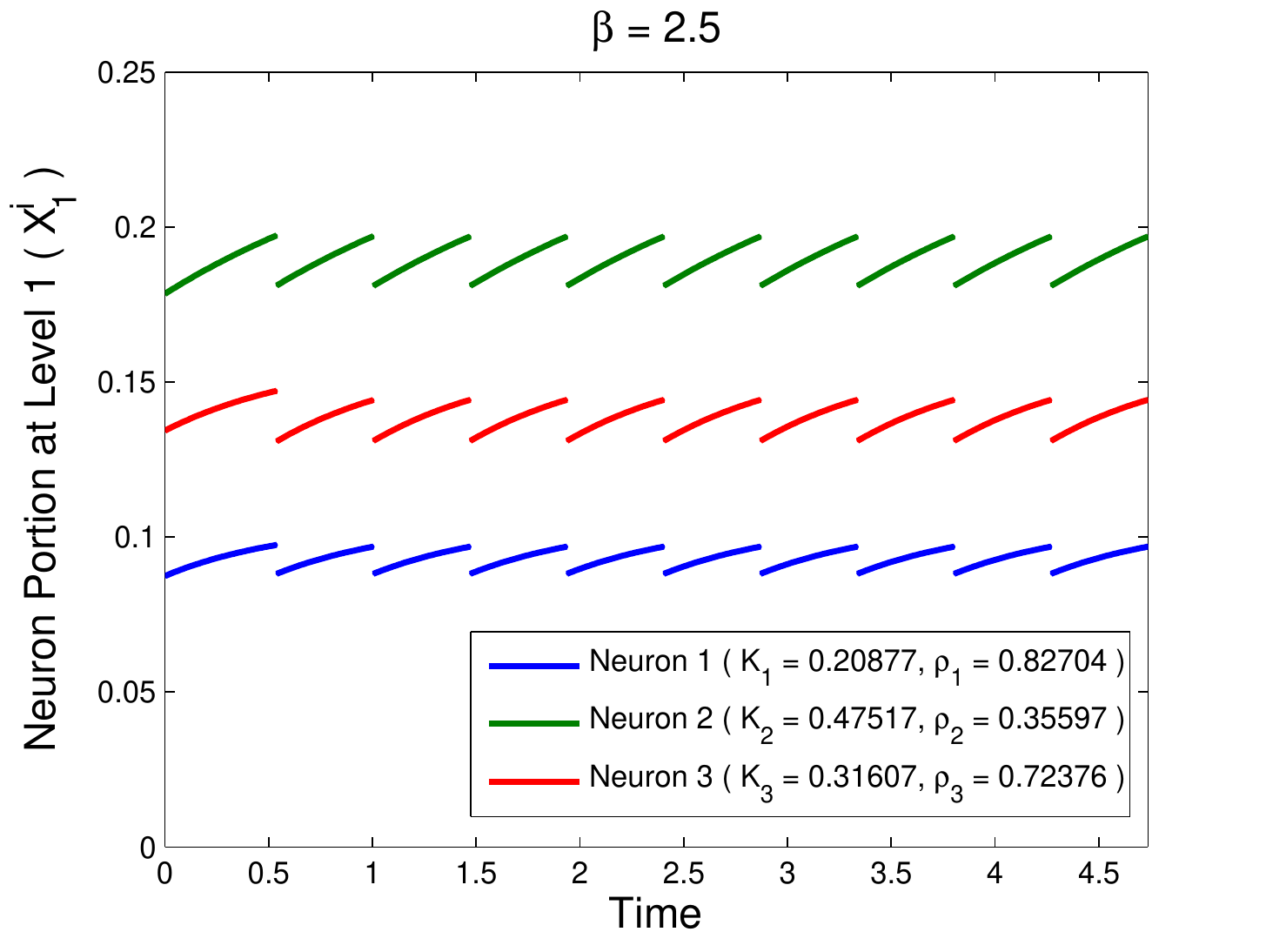}
\end{tabular}
\end{center}
\caption{Plots of the hybrid ODE-mapping system numerical simulation
  results with $\beta=2.1$ (left) and $\beta=2.5$ (right). Both of
  them are with three neuron populations. The neuron portions at
  energy level $1$ over simulation time are shown in the plots.}
\label{fig:ThreePopulations}
\end{figure}

To further demonstrate convergence, we also plot trajectories for the
same parameters for various initial conditions in
Figure~\ref{fig:ThreePopulationsConvergence}.  We see that three to
four bursts, he trajectories converge to the same periodic orbit.

\begin{figure} [ht]
\begin{center}
\begin{tabular}{ c c}
\includegraphics[width=0.5 \columnwidth]{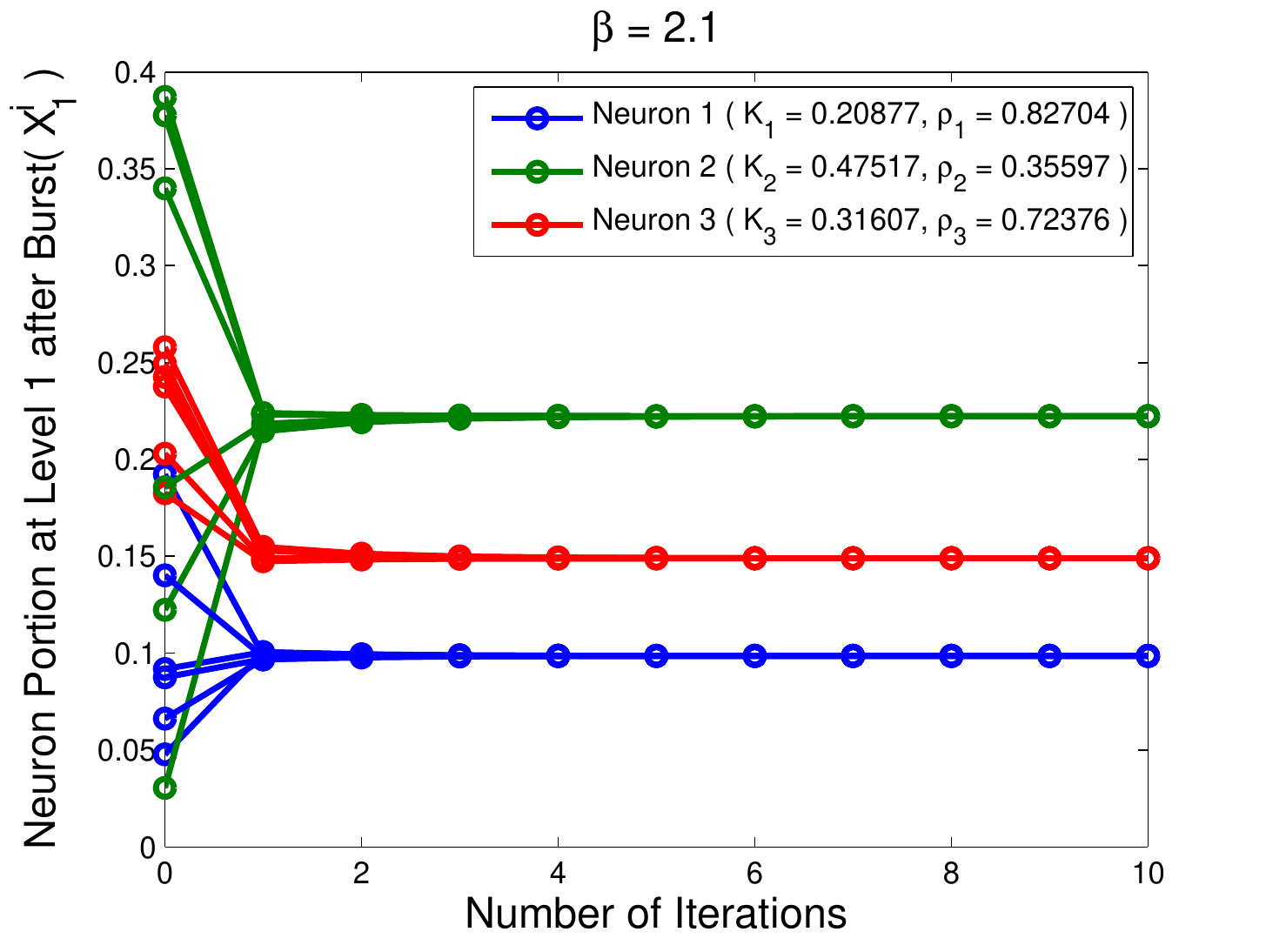} & \includegraphics[width=0.5 \columnwidth]{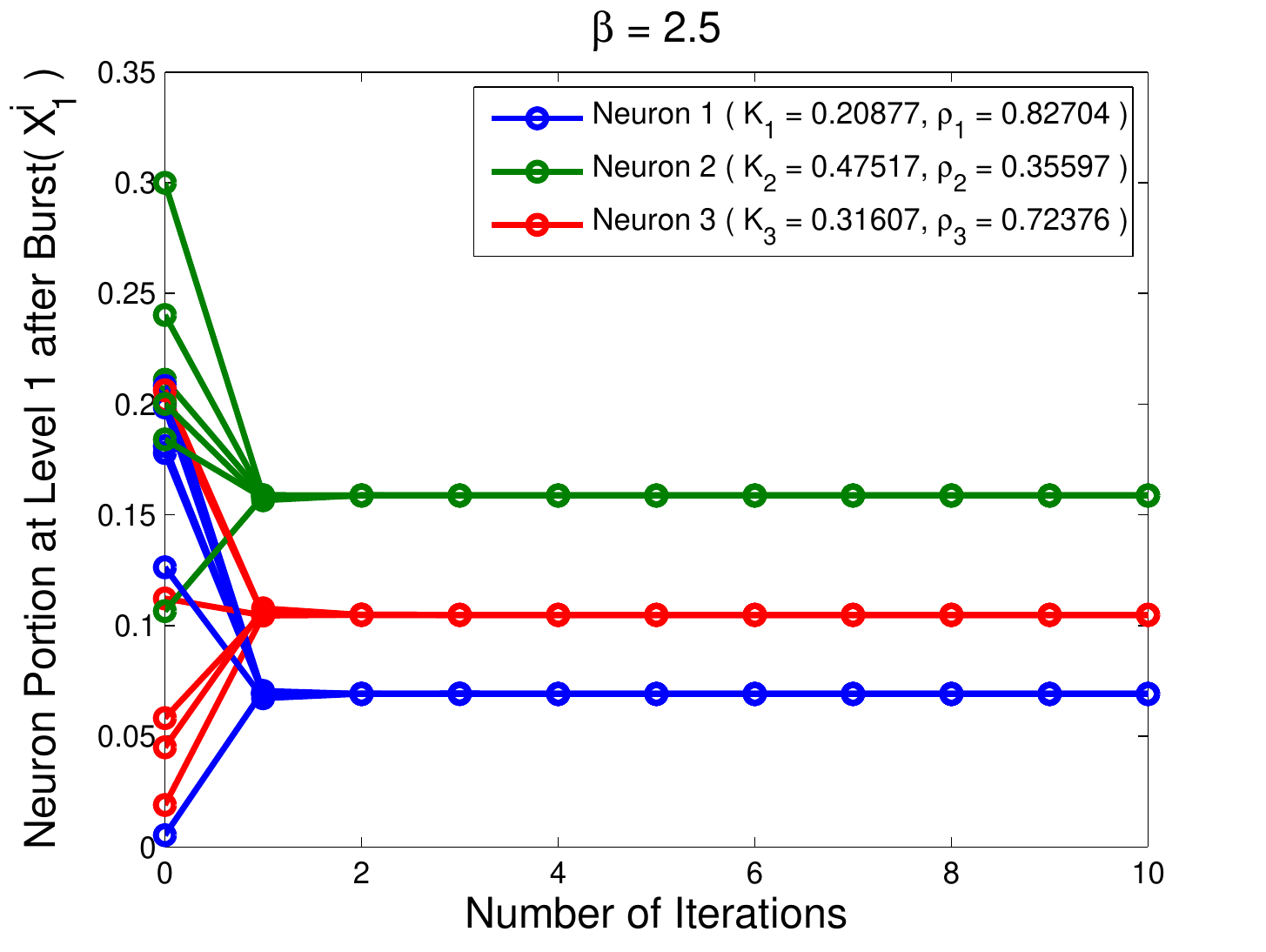}
\end{tabular}
\end{center}
\caption{Plots of neuron proportions after each burst iteration with
  $\beta=2.1$ (left) and $\beta=2.5$ (right). Both subfigures are for
  $M=3$.  For all initial conditions, the population seems to converge
  after about four bursts.}
\label{fig:ThreePopulationsConvergence}
\end{figure}

We also study the phase diagram for different $M$, with $\beta$ over
the range $[2.005, 2.5]$.  In the results above, we have only showed
that the system converges to the attractor for $\beta>\beta_M$, where
$\beta_M$ might be larger than 2.  The numerical evidence
in~\ref{fig:PhaseDiagram} suggests that $\beta_M$ might in fact be 2
in general; what we did was choose $1000$ initial conditions at
random, and plotted the proportion that fell into each of three
categories: those that converged monotonically to a periodic orbit,
those that converged non-monotonically to the periodic orbit, and
finally, those that did not converge to the periodic orbit.  (By
converge monotonically, what we mean is that each successive iteration
applied to the initial condition was monotonically convergent to the
limit and did not overshoot; by non-monotone we mean that the
iterations overshot the fixed point.)  The third case was always
empty, and the only distinction was whether the convergence was
monotone or not.

\begin{figure}[ht]
\centering
\includegraphics[width=0.45\textwidth]{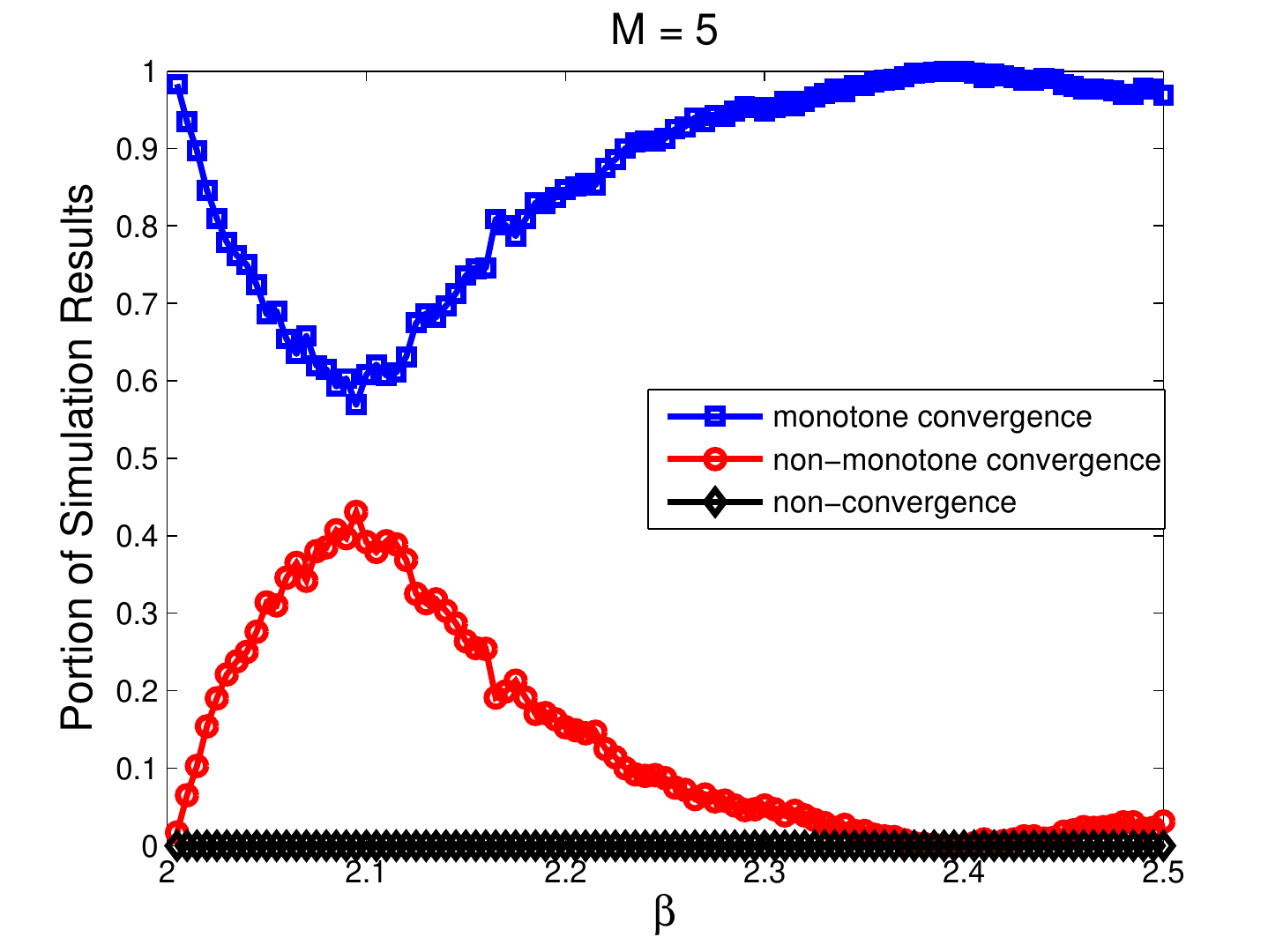}%
\includegraphics[width=0.45\textwidth]{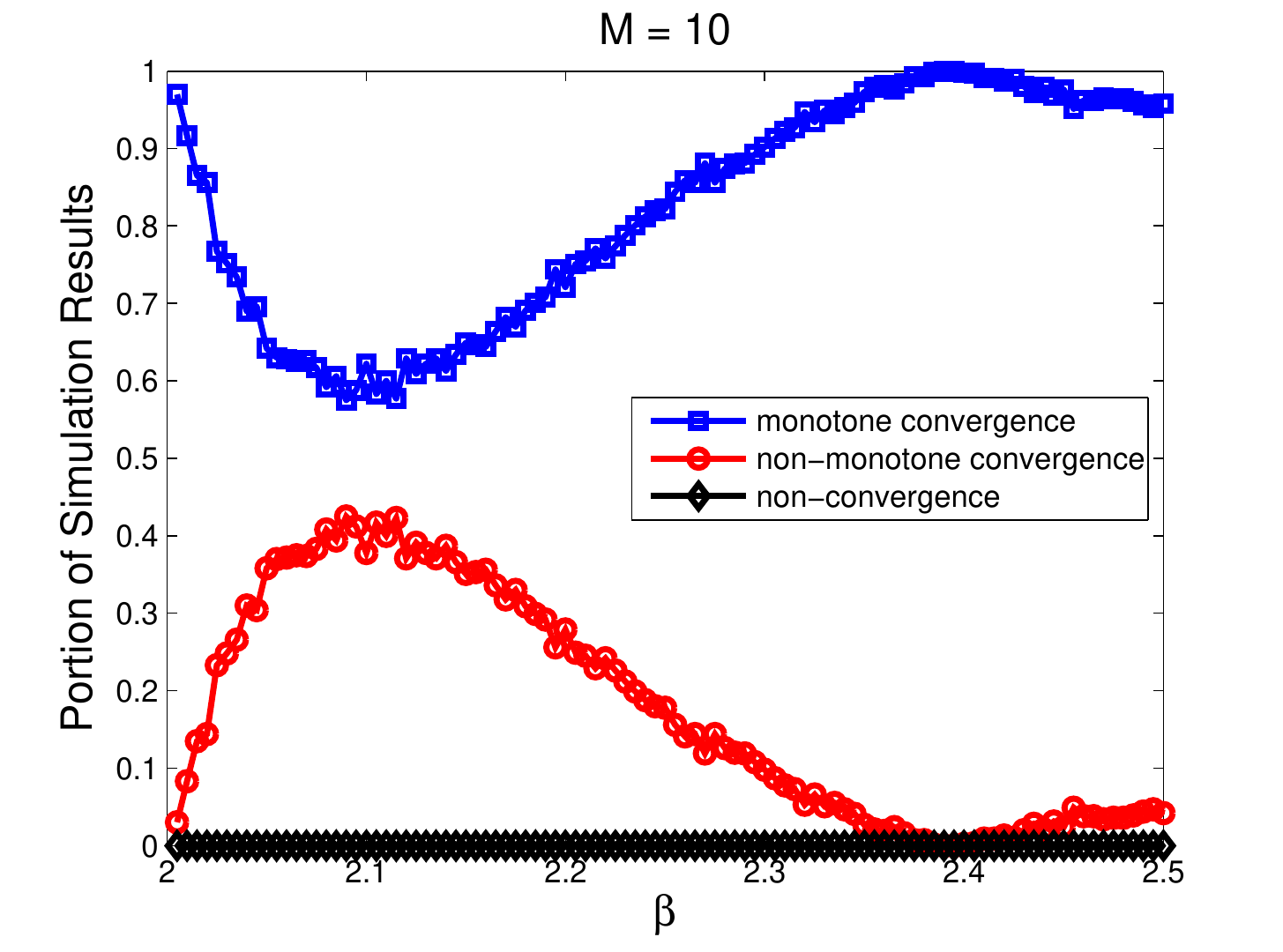}
\caption{Phase diagrams for $M=5$ and $M=10$ subpopulations. The
  parameters $K$ and $\rho$ are chosen at random. For each $\beta,$ we
  choose $10000$ initial conditions uniformly in the simplex, and
  determine which proportion falls into each of three categories:
  monotone convergent, non-monotone convergent and non-convergent. We
  vary $\beta$ from $2.005$ to $2.5$.}
\label{fig:PhaseDiagram}
\end{figure}

\section{Conclusion}

We extended the results of~\cite{BMB-DP, MMNP-DPS} to the case of
multiple subpopulations with different intrinsic firing rates.  We
were able to show that the stochastic neuronal network converges to a
mean-field limit in general.  We further analyzed the limiting mean
field in the case where each neuron has at most two inactive states,
and proved that for sufficiently large coupling parameters, the
mean-field limit has a globally attracting limit cycle. A natural next
question to ask is what happens when the system has more inactive
states, although the analysis of this higher-dimensional hybrid system
is likely to be more difficult (in analogy to the single firing rate
case of~\cite{MMNP-DPS}, where the analysis of the mean-field limit
was difficult when each neuron had many inactive states).

\section*{Acknowledgments}

Y.Z. was partially supported by the National Science Foundation
Graduate Research Fellowship under Grant No. 1122374.  L.D. was
supported by the National Science Foundation under grants CMG-0934491
and UBM-1129198 and by the National Aeronautics and Space
Administration under grant NASA-NNA13AA91A.


\end{document}